\documentclass[11pt,reqno]{amsart}
\usepackage{a4wide}
\usepackage{amsfonts,amsmath,amssymb,amscd}
\usepackage{amsthm}
\usepackage{amsmath,todonotes}
\usepackage{amscd}
\usepackage{cancel}
\usepackage{verbatim}
\usepackage{color}
\usepackage[all]{xy}
\usepackage[mathscr]{eucal}
\usepackage{mathrsfs}
\usepackage{fullpage}
\usepackage{enumitem}
\usepackage{hyperref}
\usepackage{bbm}
\usepackage{qtree}
\usepackage{mathrsfs}
\usepackage{bm}
\usepackage{bigints}
\usepackage{url}
 \usepackage[numbers]{natbib}
\allowdisplaybreaks

\usepackage[headheight=110pt,top=0.6in, bottom=0.7in, left=0.6in, right=0.6in]{geometry}
\numberwithin{equation}{section}

\newtheorem{remark}{Remark}

\newtheorem{theorem}{Theorem}

\newtheorem{lemma}[theorem]{Lemma}

\newtheorem{corollary}[theorem]{Corollary}

\newtheorem{proposition}[theorem]{Proposition}

\theoremstyle{definition}

\numberwithin{theorem}{section} 
\numberwithin{equation}{section}
\numberwithin{table}{section}

\makeatletter
\def\proof{\@ifnextchar[{\@oproof}{\@nproof}}
\def\@oproof[#1][#2]{\trivlist\item[\hskip\labelsep\textit{#2 Proof of\
		#1.}~]\ignorespaces}
\def\@nproof{\trivlist\item[\hskip\labelsep\textit{Proof.}~]\ignorespaces}

\begin{document}
\title[]{Character analogues of Cohen-type identities and related Voronoi summation formulas} 
   \author{Debika Banerjee}
\address{Debika Banerjee\\ Department of Mathematics\\
Indraprastha Institute of Information Technology IIIT, Delhi\\
Okhla, Phase III, New Delhi-110020, India.} 
\email{debika@iiitd.ac.in}
\author{Khyati Khurana}
\address{Khyati Khurana\\ Department of Mathematics\\
Indraprastha Institute of Information Technology IIIT, Delhi\\
Okhla, Phase III, New Delhi-110020, India} 
\email{khyatii@iiitd.ac.in}
 \maketitle
 \begin{abstract} 
In \cite{MR2221114}, B.~C.~Berndt and A.~Zaharescu introduced the twisted divisor sums associated with the Dirichlet character while studying the Ramanujan's type identity involving finite
trigonometric sums and doubly infinite series of Bessel functions. Later, in a follow-up paper \cite{MR3541702}, S. Kim extended the definition of the twisted divisor sums to twisted sums of divisor functions. In this paper, we derive identities associated with the aforementioned weighted divisor functions and the modified $K$-Bessel function in light of recent results obtained by the first author and B. Maji \cite{ debika2023}. Moreover, we provide a new expression for $L(1, \chi)$ from which we establish the positivity of $L(1, \chi)$ for any real primitive character $\chi$. In addition, we deduce Cohen-type identities and then exhibit the Vorono\"i-type summation formulas for them.
 \end{abstract}
	\section{Introduction}
We begin by reminiscing about a beautiful identity due to Ramanujan involving the $K$-Bessel function, which is recorded on page 253 of his lost notebook. If $\alpha$ and $\beta$ are any two positive numbers such that $\alpha \beta=\pi^2$ and  $\nu$ is any complex number, then  
  \begin{align} \label{ramanujanidentity1}
& \sqrt{\alpha}\sum_{n=1}^{\infty} \sigma_{-\nu}(n) n^{\nu/2} K_{\nu/2}(2n \alpha)\ -\sqrt{\beta}\sum_{n=1}^{\infty} \sigma_{-\nu}(n) n^{\nu/2} K_{\nu/2}(2n \beta) \notag\\
 & =  \frac{1}{4}\Gamma\left(\frac{\nu}{2}\right)\zeta(\nu)\{ \beta^{(1-\nu)/2} - \alpha^{(1-\nu)/2}\} 
 +\ \frac{1}{4}\Gamma\left(-\frac{\nu}{2}\right)\zeta(-\nu)\{ \beta^{(1+\nu)/2} - \alpha^{(1+\nu)/2} \},
 \end{align}
  where $\sigma_k(n)=\sum_{d|n}d^k$ and $K_\nu(z)$ denotes the modified Bessel function of order $\nu$ \cite[p.~78]{MR1349110}, which is defined as the following 
   \begin{align}\label{Kbessel}
  & K_{\nu}(z):=\frac{\pi}{2}\frac{I_{-\nu}(z)-I_{\nu}(z)}{\sin{\pi \nu}}, \ z\in \mathbb{C}, \nu \notin \mathbb{Z},
  \end{align}
  with $I_{\nu}$ being the Bessel function of the imaginary argument \cite[p.~77]{MR1349110} given by  
  \begin{align}\label{Inu}
  & I_{\nu}(z):=\sum_{n=0}^\infty \frac{(\frac{1}{2}z)^{\nu+2n}}{n! \Gamma{(\nu+n+1)}}, \ z\in \mathbb{C}.
  \end{align}
   Later in 1955, Guinand \cite{MR70661} derived a formula almost similar to  \eqref{ramanujanidentity1} by appealing to a formula due to Watson \cite{watson1931some} involving the $K$-Bessel function. One can use Ramanujan's formula \eqref{ramanujanidentity1} to derive Koshliakov's formula \cite{koshliakov1929voronoi}, given by
\begin{align} \label{ramanujan identity2}
\sqrt{\alpha} \left( \frac{1}{4}\gamma-\frac{1}{4}\log(4\beta)+\sum_{n=1}^\infty d(n)K_0  (  2 n\alpha )\right) =\sqrt{\beta} \left( \frac{1}{4}\gamma-\frac{1}{4}\log(4\alpha)+\sum_{n=1}^\infty d(n)K_0  (  2 n\beta )\right),
\end{align} 
where $\gamma$ denotes Euler's constant and $K_0(z)$ is defined by the limit
\begin{align}\label{K0defi}
K_0(z):=\lim_{\nu \rightarrow 0}K_\nu(z). 
\end{align}
Koshliakov, in 1929, proved the formula \eqref{ramanujan identity2} by employing the Vorono\"i summation formula \cite{voronoi1904fonction}, which reads as the following
\begin{align} \label{vor}
 \sum_{a \leq n\leq b} \vspace{-.2cm}^\prime d(n)f(n)=\int_{a}^b(\log (x)+2\gamma)f(x) d\it{x} 
 +\sum_{ n=1}^\infty d(n) \int_{a}^b f(x)\ (4 K_0(4\pi \sqrt{nx})- 2 \pi Y_0(4\pi \sqrt{nx}) )d\it{x},
 \end{align}
 where the prime $^\prime$ on the summation of the left-hand side implies that if $a $ or $b$ is an integer, only $f(a)/2$ or $f(b)/2$ is counted, respectively. Here $f(x)$ is a function of bounded variation in $(a,b)$ with $0<a<b$, and $K_0(z)$ is defined in \eqref{K0defi}, and $Y_\nu(z)$ denotes the Weber-Bessel function of order $\nu$ \cite[p.~ 64]{MR1349110} given by 
 \begin{align}\label{Y_bessel}
  & Y_{\nu}(z):= \frac{J_{\nu}(z)\cos{\pi \nu}-J_{-\nu}(z)}{\sin{\pi \nu}}, \ z\in \mathbb{C}, \nu \notin \mathbb{Z},\\
  &Y_n(z):=\lim_{\nu \rightarrow 0}Y_\nu(z), \   z\in \mathbb{C}, \ n \in \mathbb{Z}, 
  \end{align} 
  and $J_\nu(z)$ denotes the Bessel function of the first kind of order $\nu$ \cite[p.~ 40]{MR1349110} 
  \begin{align}\label{J_bessel}
  J_{\nu}(z):=\sum_{n=0}^\infty \frac{(-1)^n(\frac{1}{2}z)^{\nu+2n}}{n! \Gamma{(\nu+n+1)}}, \ z\in \mathbb{C}.
  \end{align}
   After Vorono\"i's remarkable discovery of \eqref{vor}, many number theorists examined the formula \eqref{vor} and provided proofs under different conditions on the function $f(x)$. A.~L.~Dixon and W.~L.~Ferrar \cite{dixon1931lattice} gave proof for a bounded second differential coefficient function $f(x)$ in $(a,b)$. Koshliakov proved \eqref{vor} for the analytic function $f$ inside a closed contour strictly containing the interval $[a,b]$ with $0<a <b$. J.~R.~Wilton \cite{wilton1932voronoi} proved \eqref{vor} for the function $f$, which has compact support in the interval $[a,b]$ such that $ \lim_{\varepsilon \rightarrow 0}V_\alpha^{\beta-\varepsilon} f(x)=V_\alpha^{\beta-0} f(x)$ where $V_\alpha^{\beta }$ denotes the total variation of $f(x)$ over $(\alpha,\beta)$. In 1987, M. Jutila \cite{MR910497} gave a Vorono\"i-type summation formula involving an exponential factor. One can refer to  \cite{MR1642870, MR0352030, MR0280693, MR1211020} for details and developments on Vorono\"i's summation formulas. Apart from its connection to different fields of mathematics, Vorono\"i-type summation formulas also have some applications in physics, especially in quantum graph theory \cite{MR2788722}.  
\par
 After Koshliakov, many mathematicians studied his formula \eqref{ramanujan identity2}. In 1936, Ferrar \cite{MR1574771} reproved \eqref{ramanujan identity2} by appealing to the functional equation of $\zeta(s)$. Later in 1966, K. Soni \cite{MR201400} showed that the functional equation of $\zeta^2(s)$ is equivalent to the Vorono\"i summation formula \eqref{vor} and is equivalent to Koshliakov's formula \eqref{ramanujan identity2}. In 1972, Oberhettinger and Soni \cite{MR308065} established that the functional equation of $\zeta(s)$ and Koshliakov's formula are equivalent using the methods of Hamburger. In 2008, B.~C.~Berndt, Y.~Lee, and J.~Sohn \cite{MR2462944} proved  \eqref{ramanujanidentity1} by elaborating Guinand's method. They rediscovered Koshliakov's formula \eqref{ramanujan identity2} by taking $  {\nu\to  0}$ in \eqref{ramanujanidentity1}.
However, A.~Dixit in \cite{MR2837111} gave an extended version of Ramanujan's formula \eqref{ramanujanidentity1} by appealing to the Cauchy residue theorem and the theory of the Mellin transform. Further analysis of identities analogous to \eqref{ramanujanidentity1} and \eqref{ramanujan identity2} have been done by B.~ C.~Berndt, S.~Kim and A.~Zaharescu in \cite{MR3307490}. They studied character analogues of Koshliakov's formula \eqref{ramanujan identity2} for even characters. They replaced the classical divisor function $d(n)$ with the twisted divisor sums, namely,
\begin{align}\label{defi1}
    d_{\chi}(n)=\sum_{d|n} \chi(d),  \ \  \ \ \ \ \ \ \ d_{\chi_1,\chi_2}(n)=\sum_{d|n} \chi_1(d)\chi_2(n/d),
\end{align} 
where $\chi, \chi_1$ and $\chi_2$ are the Dirichlet characters, and they proved the following beautiful identity
 \begin{align*}
 \frac{qL(1,\chi) }{4\tau(\chi)}+\sum_{n=1}^\infty d_{\chi } (n) K_0 \left(\frac{2 \pi nz}{\sqrt{q}}\right)= \frac{\sqrt{q}L(1,\chi) }{4z} +\frac{\tau(\chi)}{z\sqrt{q}}\sum_{n=1}^\infty d_{\Bar{\chi}  } (n) K_0 \left(\frac{2 \pi n}{z\sqrt{q}}\right),
\end{align*}
where $\chi$ is a non-principal even primitive character mod $q$, $\Re(z)>0$, and $\tau(\chi)$ is the Gauss sum defined in \eqref{Gauss},
and $K_0(z)$ is defined in \eqref{K0defi}. In particular, for even real character  $\chi$, they established the positivity of $L(1, \chi)$, which is instrumental in proving Dirichlet’s theorem on primes in arithmetic progressions.
The weighted divisor sums defined in \eqref{defi1} were introduced by B.~C.~Berndt and A.~Zaharescu \cite{MR2221114}, where they showed that the twisted or weighted divisor sums could be studied in connection with identities associated with $r_2(n)$. However, S.~Kim \cite{MR3541702} extended the definition of twisted divisor sums to twisted sums of divisor functions, namely, 
\begin{align}\label{defimainpaper}
   \sigma_{k,\chi}(n):=\sum_{d|n}d^k  \chi(d), \ \ \ \ \bar{\sigma}_{k,\chi}(n):=\sum_{d|n}d^k  \chi(n/d), \ \ \ \ \sigma_{k,\chi_1,\chi_2}(n):=\sum_{d|n}d^k  \chi_1(d) \chi_2(n/d),
\end{align}
  and they studied Riesz sum-type identities associated with them. Recently  A.~Dixit and A.~Kesarwani \cite{MR3730446} studied a new generalization of the modified Bessel function of the second kind. They derived a formula analogous to \eqref{ramanujanidentity1} associated with the generalized Bessel function. They proved that their formula is equivalent to the functional equation of a non-holomorphic Eisenstein series on  $SL(2, \mathbb{Z})$.
\par  The study of the infinite series in \eqref{ramanujanidentity1} is of prime importance as it is intimately connected with the Fourier series expansion of non-holomorphic Eisenstein series on $SL(2, \mathbb{Z})$ or Maass wave forms \cite{MR31519,MR2135107,MR0429749,MR791406}. Motivated by this fact, Cohen, in 2010 \cite{MR2744771}, established the following result, similar to \eqref{ramanujanidentity1},
\begin{align}\label{guinand}
  4x^{\frac{1}{2}}\sum_{n=1}^{\infty}  \frac{\sigma_{\nu}(n)}{n^{\nu/2}}   K_{\nu/2}(2\pi n x)\ +\   \Lambda(s)(x^{(1-\nu)/2} - x^{(\nu-1)/2} )& = 4x^{-\frac{1}{2}}\sum_{n=1}^{\infty}  \frac{\sigma_{\nu}(n)}{n^{\nu/2}}  K_{\nu/2}( \frac{2 \pi n}{x})\notag\\
 &+\   \Lambda(-s)( x^{-(1+\nu)/2} - x^{(1+\nu)/2}), 
 \end{align}
where $\Lambda(s)= \pi^{-\frac{s}{2}}\Gamma\left(  \frac{s}{2}\right)\zeta(s)$ and $K_\nu(z)$ is defined in \eqref{Kbessel}. As an application, he obtained the following beautiful identity involving the divisor function $\sigma_s(n)$ and the modified $K$-Bessel function.
\begin{proposition}\cite[p.~62, Theorem 3.4]{MR2744771} \label{Cohen-type} For $\nu \notin \mathbb{Z}$ such that $\Re(\nu) \geq 0$ and any integer N such that $N \geq \lfloor \frac{\Re(\nu)+1}{2}\rfloor$, then 
  \small         
        \begin{align}
     &8 \pi x^{\frac{\nu}{2}}\sum_{n=1}^{\infty} \sigma_{-\nu}(n) n^{\nu/2} K_{\nu}(4\pi \sqrt{nx})
    =-\frac{ \Gamma(\nu) \zeta(\nu)}{(2\pi)^{\nu-1} }  +  \frac{\Gamma(1+\nu) \zeta(1+\nu)}{\pi^{\nu+1} 2^\nu x} + \left\{ \frac{\zeta(\nu)x^{\nu-1}}{\sin\left(\frac{\pi \nu}{2}\right)}  + \frac{2}{  \sin \left(\frac{\pi \nu}{2}\right)}\sum_{j=1}^{N} \zeta(2j)\ \zeta(2j-\nu)x^{2j-1} 
    \right.\notag\\&\left.\ \  
     \hspace{5cm}-\pi\frac{\zeta(\nu+1) x^{\nu}} {\cos(\frac{\pi\nu}{2}) }  +\frac{2}{  \sin \left(\frac{\pi \nu}{2}\right)}\sum_{n=1}^{\infty}{\sigma}_{-\nu}(n)\frac{x^{2N+1} }{    (n^2-x^2) } \left( n^{\nu-2N}-x^{\nu-2N}\right)
    \right\}. \label{cohen}
    	 \end{align} 
      \end{proposition}
In addition to \eqref{cohen}, he derived several interesting identities involving the divisor function $\sigma_s(n)$ and the modified $K$-Bessel function. Later, B.~C.~Berndt, A.~Dixit, A.~Roy, and A.~Zaharescu \cite{MR3558223}, in their seminal work, showed that Cohen-type identity \eqref{cohen} can be used to derive the Vorono\"i-type summation formula for $\sigma_s(n)$.
 \begin{proposition} \cite[p.~841, Theorem 6.1]{MR3558223}\label{vorlemma1}
Let $0< \alpha< \beta$ and $\alpha, \beta \notin \mathbb{Z}$. Let $f$ denote a function analytic inside a closed contour strictly containing $[\alpha, \beta ]$. Assume that $-\frac{1}{2}< \Re{(\nu)}<\frac{1}{2}$.  Then 
           \begin{align*}     
          &  \sum_{\alpha<j <\beta}  {\sigma}_{-\nu}(j)f(j)   =   \int_{\alpha} ^{  \beta   }f(t) \left(   \zeta(1-\nu, \chi) \ t^{-\nu}   +    \zeta(\nu+1)         \right)  dt  \notag \\
&+ 2\pi \sum_{n=1}^{\infty} \sigma_{-\nu}(n) n^{\nu/2}    \int_{\alpha} ^{  \beta   }f(t)  (t)^{-\frac{\nu}{2}}   \left\{  \left(   \frac{2}{\pi} K_{\nu}(4\pi   \sqrt{nt}) - Y_{\nu}(4\pi   \sqrt{nt})\right) \cos \left(\frac{\pi \nu}{2}\right) 
  -  J_{\nu}(4\pi   \sqrt{nt})  \sin \left(\frac{\pi \nu}{2}\right)  \right\}  dt.
           \end{align*}
\end{proposition}
 Inspired by Cohen's results \cite{MR2744771}, the first author and B. Maji \cite{debika2023} studied the infinite series involving the generalised divisor function and the modified $K$-Bessel functions. More precisely, they studied the following infinite series, for $r\in \mathbb{Z}, z \in \mathbb{C}$ and $a$ and $x$ be any two positive real numbers,
         \begin{align}\label{Deb-Maji}
      \sum_{n=1}^\infty\sigma_{z}^{(r)}(n)n^{\frac{\nu}{2}}K_\nu(a\sqrt{nx}),   
      \end{align}
   where $\sigma_{z}^{(r)}(n)=\sum_{d^r|n}d^z$ and $\nu$ is a complex number with $\Re(\nu)\geq 0$. It is important to note that $\sigma_{z}^{(1)}(n)=\sigma_{z}(n)$. Hence almost all the Cohen-type identities can be derived from their results. 
   In this article, we are interested in the character analogues of \eqref{Deb-Maji}. That is, we study the following infinite series
    \begin{align}\label{series defi}
   \sum_{n=1}^\infty\sigma_{z,\chi }(n)&n^{\frac{\nu}{2}}K_\nu(a\sqrt{nx}),  \ \ \ \ \ \ \ \ \ \   \sum_{n=1}^\infty \bar{\sigma}_{z,\chi }(n)n^{\frac{\nu}{2}}K_\nu(a\sqrt{nx}), \ \ \ \ \ \ \ \ \ 
   \sum_{n=1}^\infty\sigma_{z,\chi_1,\chi_2}(n)n^{\frac{\nu}{2}}K_\nu(a\sqrt{nx}),   
  \end{align}
   where $\sigma_{z,\chi }(n)$, $\bar{\sigma}_{z,\chi }(n)$ and $\sigma_{z,\chi_1,\chi_2}(n)$ are defined in \eqref{defimainpaper}, and $\nu$ is a complex number with $\Re(\nu)\geq 0$.  We derive Cohen-type identities for twisted sums of divisor functions $\sigma_{z,\chi }(n)$, $\bar{\sigma}_{z,\chi }(n)$  and $\sigma_{z,\chi_1,\chi_2}(n)$ and obtain the Vorono\"i-type summation formula for them. 
The paper is organized as follows:		 
  Section \ref{integer results} states the results for the twisted sums of divisor functions when $z \in\mathbb{Z}$. Section \ref{cohen identities...} provides Cohen-type identities for them. Section \ref{voronoi identities...} states the Vorono\"i-type summation formula for twisted sums of divisor functions. Section \ref{preliminary} reviews several significant results needed to derive our main results. Sections \ref{proof of integer nu}, \ref{proof of cohen identities...} and \ref{proof of voronoi...} are devoted to the proofs of identities stated in Sections \ref{integer results}, \ref{cohen identities...} and \ref{voronoi identities...}, respectively. 
\section{Main Results for the case  \texorpdfstring{$z \in\mathbb{Z}_{\geq 0}$}{TEXT}   }\label{integer results}
In this section, we consider $z$ a non-negative integer and denote it by $k$. Throughout the paper, we assume that $a$ and $x$ are two strictly positive real numbers, which differ from $0$. Before proceeding further, we will mention some definitions and notations which will be used later.
\par The Dirichlet $L$-function is defined by
\begin{align}\label{Lfunctiondefi}
L(s,\chi):=\sum_{n=1}^{\infty}\frac{\chi(n)}{n^s}, \ \ \Re(s)>1,
\end{align}
where $\chi$ is a Dirichlet character modulo $q$. It can be meromorphically  continued to the entire complex plane. Furthermore, if $\chi$ is principal, the corresponding Dirichlet $L$-function has a simple pole at $s = 1$. Otherwise, the $L$-function is entire.
\par 
The Gauss sum of a Dirichlet character modulo $q$ 
is 
\begin{align}\label{Gauss}
\tau(\chi):=\sum_{h=1}^q \chi(h)e^{2\pi i h/q}.
\end{align}
\par Since our results involve the modified $K$-Bessel function, it is important to state some related results. The asymptotic estimate for the $K$-Bessel function defined in \eqref{Kbessel} is \cite[p.~202]{MR1349110}
\begin{align*}
 K_{\nu}(x) =\left(\frac{\pi}{2x}  \right)^{\frac12} e^{-x} +O\left(\frac{e^{-x}}{x^{\frac32}}  \right) \ \mbox{as}\ x \rightarrow \infty.
\end{align*}
The above expression ensures the absolute convergence of all the infinite series defined in \eqref{series defi}. Throughout this paper, we will consider $\Re (\nu)\geq 0$ as $K_{-\nu}(x)=K_\nu(x)$. We recall that $K_0(x)$ is defined by \eqref{K0defi}. From the integral representation of $K_0(x)$ \cite[p.~446]{MR1349110}
\begin{align*}
    K_0(x)=\int_{0}^{\infty} e^{-x \cosh t} dt,
\end{align*}
one can see that $K_0(x)$ is positive and monotonically decreasing on the interval $(0, \infty)$. We also note the series representation of $K_0(x)$ \cite[p.~80]{MR1349110} 
 \begin{align*}
K_0(x)=-\log\left(\frac{x}{2}\right) I_0(x)+\sum_{m=0}^\infty\frac{\left(\frac{x}{2}\right)^{2m}}{(m!)^2 } \frac{\Gamma^\prime(m+1)}{\Gamma(m+1)},
 \end{align*}
 where $I_0(x)$ is defined in \eqref{Inu}. From its series representation  mentioned above, one can infer that $K_0(x)$ tends to $+\infty$ as $x$ decreases to $0$.
\subsection{Identities involving odd characters.} In this subsection, we will consider $k$ to be an even, non-negative integer and $\chi$ an odd primitive character.
			\begin{theorem}\label{thm1}
 				Let $k$ be an even, non-negative integer and $\chi$ be an odd primitive Dirichlet character modulo $q$. Then, for any $\Re{(\nu)}>0$, 
 \begin{align*}
   \sum_{n=1}^\infty\sigma_{k,\chi}(n)n^{\frac{\nu}{2}}K_\nu(a\sqrt{nx})=& {\delta_{k}}   \frac{2^{\nu+1}}{a^{\nu+2}} \Gamma(1+\nu)L(1,\chi)x^{-\frac{\nu}{2}-1}+\frac{(-1)^{\frac{k}{2}}i  q^k  }{a^\nu 2^{k+2-\nu}  \pi^{k+1}}   \Gamma(\nu)\tau(\chi) \Gamma(k+1)L(k+1,\Bar{\chi})\ x^{-\frac{\nu}{2}} \notag\\
					&-\frac{(-1)^{\frac{k}{2}}i a^{\nu}q^{\nu+k} \ x^{\frac{\nu}{2}} }{2^{3\nu+k+2} \pi^{2\nu+k+1}}\Gamma(\nu+k+1)\tau(\chi)  \sum_{n=1}^\infty   \frac{ \Bar{\sigma}_{k,\Bar{\chi}}(n)}{\left( n+\frac{a^2qx}{16\pi^2} \right)^{\nu+k+1} },
				\end{align*} 
where $\delta_k$ is given by 
\begin{align}
\delta_{k}=\begin{cases}
          1,   & \ \ \text{if } k=0,\\
          0, & \ \ \text{if } k>0.
       \end{cases} \label{d_k number}
\end{align} 
			\end{theorem}
Our next result corresponds to $\nu=0$ is as follows
 \begin{theorem}\label{odd_k0thm1}
Let $k$ be an even, non-negative integer and $\chi$ be an odd primitive Dirichlet character modulo $q$. Then  
   \begin{align}\label{Positivity1}
				 	\sum_{n=1}^\infty\sigma_{k,\chi}(n) K_0(a\sqrt{nx})=&\delta_k \frac{2}{a^2x} L(1,\chi)- \frac{ L(-k,\chi) }{4}\left(  \log \left(\frac{8\pi}{a^2}   \right) +\frac{ L^\prime(-k,\chi)}{ L(-k,\chi)} -2\gamma \right)+\frac{ L(-k,\chi) }{4} \log x\notag\\
      &+(-1)^{\frac{k}{2}}\frac{ik!q^k}{2 ({2\pi })^{k+1} }  \tau(\chi)     \sum_{n=1}^\infty {\bar{\sigma}_{k, \bar{\chi}}}(n)   \left(\frac{1}{n^{k+1}}-\frac{1}{( n+\frac{a^2qx}{16\pi^2} )^{k+1} } \right),
  \end{align} 
where  $\delta_k$ is defined in \eqref{d_k number}.
 \end{theorem}
\begin{remark}\label{remark1}
Let us assume that $\chi$ is a real odd primitive Dirichlet character modulo $q$. Now setting $k=0$ and then employing the functional equation \eqref{ll(s)} in \eqref{Positivity1}, we obtain
 \begin{align}\label{lo}
        \sum_{n=1}^\infty d_\chi(n) K_0(a\sqrt{nx})=& \frac{L(1,\chi)}{x} \left( \frac{2}{a^2}  -\frac{i\tau(\chi)}{4\pi }   x \log x\right)- \frac{ L(0,\chi) }{4}\left(  \log \left(\frac{8\pi}{a^2}   \right) +\frac{ L^\prime(0,\chi)}{ L(0,\chi)} -2\gamma \right) \notag\\
        &+ \frac{i a^2q \ x}{ {64\pi^3 }  }  \tau(\chi)     \sum_{n=1}^\infty\frac{{d_{\bar{\chi}}}(n)  }{n( n+\frac{a^2qx}{16\pi^2} )}. 
    \end{align}  
  Now we can easily show that $d_{\chi}(n)$ is non-negative for each $n$ from the Euler product on the left-hand side of \eqref{Lfz1}. More precisely, the factors in its Euler product are of the forms $$\left(1-\frac{1}{p^{s}}\right)^{-1}, \ \left(1-\frac{1}{p^{s}}\right)^{-2}, \ \mbox{or} \ 
     \left(1-\frac{1}{p^{2s}}\right)^{-1}, $$ 
     according as to whether $\chi(p)=0, 1$ or $-1$ respectively. Therefore, by rewriting the Euler product as a Dirichlet series, one can easily notice that $d_{\chi}(n)\geq 0$ for all $n$. In addition, it is clear from \eqref{defi1} that $d_{\chi}(n) \geq 1$ whenever $n$ is a perfect square.
     We have already mentioned the fact that $K_0(x)$ tends to $+\infty$ as $x$ decreases to $0$ at the beginning of this section. Therefore, the left-hand side of \eqref{lo} approaches $+\infty$ as $x$ decreases to $0$. Let us examine the right-hand side of \eqref{lo}. Noting that $i\tau(\chi)$ is real for real odd primitive Dirichlet character \cite[Theorem 9.9, p.~288]{MR2378655}, we can easily deduce that the infinite series on the right-hand side of \eqref{lo} tends to $0$ as $x$ decreases to $0$. Next noting that $i\tau(\chi)$ is real and $x \log x $ tends to $0$ as $x$ decreases to $0$, we infer that $\frac{ L(1,\chi)}{x}$ tends to $+\infty$ as $x$ decreases to $0$, which ensures the strict positivity of $L(1, \chi)$.
\end{remark}
			\begin{theorem}\label{thm3}
				Let  $k\geq 2$ be an even integer and $\chi$ be an odd primitive Dirichlet character modulo $q$. Then, for any  $\Re{(\nu)}>0$,
    \begin{align*}
   \sum_{n=1}^\infty\Bar{\sigma}_{k,\chi}(n)n^{\frac{\nu}{2}}K_\nu(a\sqrt{nx})=& 	\frac{2^{\nu+2k+1}}{a^{\nu+2k+2}}  \Gamma(k+1)\Gamma(\nu+k+1)L(1+k,\chi)\ x^{-\frac{\nu}{2}-k-1} \notag\\ &-\frac{(-1)^{\frac{k}{2}}i (aq)^\nu \ x^{\frac{\nu}{2}}}{2^{3\nu+k+2}\pi^{2\nu+k+1}}  \Gamma(\nu+k+1)\tau(\chi)  \sum_{n=1}^\infty\frac{\sigma_{k,\Bar{\chi}}(n)}{\left( n+\frac{a^2qx}{16\pi^2} \right)^{\nu+k+1} } .
\end{align*} 
    \end{theorem}
    The result corresponding to $\nu=0$ is as follows
   \begin{theorem}\label{1thm1}
Let  $k\geq 2$ be an even integer and $\chi$ be an odd primitive Dirichlet character modulo $q$. Then 
   \begin{align*}
				\sum_{n=1}^\infty \bar{\sigma}_{k,\chi}(n) K_0(a\sqrt{nx})=&\frac{ 2^{2k+1}}{ a^{2k+2}}\Gamma^2(k+1)L(k+1,\chi)\frac{1}{x^{k+1}} +\frac{1}{2}\zeta^\prime(-k)L(0,\chi) \nonumber \\&
   + \frac{(-1)^{\frac{k}{2}}ik! \tau(\chi)  }{2(2\pi)^{k+1}    }\sum_{n=1}^\infty \sigma_{k,\Bar{\chi}   }(n) \left(\frac{1}{n^{k+1}}-\frac{1}{(n+ \frac{a^2qx}{16\pi^2} )^{k+1} } \right). 
\end{align*}
 \end{theorem}
\begin{remark}
    The case $k=0$ is excluded from Theorem \ref{thm3} and Theorem \ref{1thm1} because of the fact that $\bar{\sigma}_{0,\chi}(n)=\sigma_{0,\chi}(n) =d_{\chi}(n)$ and $\sum_{n=1}^\infty d_{\chi}(n) n^{\nu/2} K_{\nu}(a\sqrt{nx})$ for $\Re(\nu)>0$ and $\sum_{n=1}^\infty d_{\chi}(n) K_0(a\sqrt{nx})$ are already considered in Theorems \ref{thm1} and  \ref{odd_k0thm1}, respectively.
\end{remark}
		\subsection{Identities involving even characters}  In this subsection, we present similar results when $k$ is an odd positive integer and $\chi $ is a non-principal even primitive character.
			\begin{theorem}\label{thmeven1}
				 Let $k\geq 1$ be an odd integer and $\chi$ be a non-principal even primitive Dirichlet character modulo $q$. Then, for any  $\Re{(\nu)}>0$, 
     \begin{align*}
   \sum_{n=1}^\infty\sigma_{k,\chi}(n)n^{\frac{\nu}{2}}K_\nu(a\sqrt{nx})=&\frac{(-1)^{\frac{k-1}{2}} q^k }{a^{\nu}2^{k+2-\nu}\pi^{k+1}}   \Gamma(\nu)\tau(\chi) \Gamma(k+1)L(1+k,\Bar{\chi})\ x^{-\frac{\nu}{2}}\notag\\
   &+\frac{(-1)^{\frac{k+1}{2}}a^\nu q^{\nu+k}\ x^{\frac{\nu}{2}}}{2^{3\nu+k+2}\pi^{2\nu+k+1}} \Gamma(\nu+k+1)\tau(\chi)  \sum_{n=1}^\infty \frac{\sigma_{k,\Bar{\chi}}(n)}{\left( n+\frac{a^2qx}{16\pi^2} \right)^{\nu+k+1} }.
\end{align*} 
     \end{theorem}
 The result corresponding to $\nu=0$ is as follows
 \begin{theorem}\label{7} 
Let $k\geq 1$ be an odd integer and $\chi$ be a non-principal even primitive Dirichlet character modulo $q$. Then  
     \begin{align*}
				 	\sum_{n=1}^\infty\sigma_{k,\chi}(n) K_0(a\sqrt{nx})=&- \frac{ L(-k,\chi) }{4}\left(  \log \left(\frac{8\pi}{a^2}   \right) +\frac{ L^\prime(-k,\chi)}{ L(-k,\chi)} -2\gamma \right)+\frac{ L(-k,\chi) }{4} \log x\notag\\
      &+(-1)^{\frac{k-1}{2}}\frac{k!q^k}{2 ({2\pi })^{k+1} }  \tau(\chi)     \sum_{n=1}^\infty {\bar{\sigma}_{k, \bar{\chi}}}(n)   \left(\frac{1}{n^{k+1}}-\frac{1}{(n+ \frac{a^2qx}{16\pi^2} )^{k+1} } \right). 
  \end{align*} 
    \end{theorem}
			\begin{theorem}\label{thmeven2}
				 Let $k\geq 1$ be an odd integer and $\chi$ be a non-principal even primitive Dirichlet character modulo $q$. Then, for any  $\Re{(\nu)}>0$,  
     \begin{align*}
   \sum_{n=1}^\infty\Bar{\sigma}_{k,\chi}(n)n^{\frac{\nu}{2}}K_\nu(a\sqrt{nx})=& 	\frac{2^{\nu+2k+1}}{a^{\nu+2k+2}} \Gamma(k+1)\Gamma(\nu+k+1)L(1+k,\chi)\  x^{-\frac{\nu}{2}-k-1}\notag\\
     &+\frac{(-1)^{\frac{k+1}{2}}(aq)^\nu \ x^{\frac{\nu}{2}}}{2^{3\nu+k+2}\pi^{2\nu+k+1}} \Gamma(\nu+k+1)\tau(\chi)  \sum_{n=1}^\infty\sigma_{k,\Bar{\chi}}(n)\frac{1}{\left( n+\frac{a^2qx}{16\pi^2} \right)^{\nu+k+1} }.
\end{align*} 
			\end{theorem}
The result corresponding to $\nu=0$ is as follows
 \begin{theorem}\label{10} 
				Let $k\geq 1$ be an odd integer and $\chi$ be a non-principal even primitive Dirichlet character modulo $q$. Then   
    \begin{align*}
				\sum_{n=1}^\infty \bar{\sigma}_{k,\chi}(n) K_0(a\sqrt{nx})=&\frac{ 2^{2k+1}}{ a^{2k+2}}\Gamma^2(k+1)L(k+1,\chi)\frac{1}{x^{k+1}} +\frac{1}{2}\zeta^\prime(-k)L(0,\chi)\\&
   + \frac{(-1)^{\frac{k-1}{2}}k!  }{2(2\pi)^{k+1}    }\tau(\chi)  \sum_{n=1}^\infty \sigma_{k,\Bar{\chi}   }(n) \left(\frac{1}{n^{k+1}}-\frac{1}{(n+ \frac{a^2qx}{16\pi^2} )^{k+1} } \right). 
\end{align*}
  \end{theorem}
  The next result corresponds to the case $\nu=0$ and $k=0$.  We can also claim the positivity of $L(1, \chi)$ for even real character $\chi$ from the following identity.
 \begin{theorem}\label{sigmanu=0}
 Let $\chi$ be a non-principal even primitive Dirichlet character modulo $q$. Then we have
 \begin{align}\label{Kosh_Bruce_Analog}
   \sum_{n=1}^\infty d_{\chi}(n) K_0(a\sqrt{nx})=   \frac{2}{a^2x}   L(1,\chi)-\frac{\tau(\chi)}{8 }L (1,\bar{\chi})  + \frac{a^2q\ x}{32\pi^4} \tau( \chi )  \sum_{n=1}^\infty d_{\Bar{\chi} }(n)
      \frac{\log \left(\frac{16\pi^2 n}{a^2q  x} \right) }{n^2-\left(\frac{a^2qx}{16\pi^2 }  \right)^2},
\end{align}
   provided $\frac{a^2 q x}{16 \pi^2} \notin \mathbb{Z}_+$.
\end{theorem}
\begin{remark}
When $\chi$ is any real even primitive Dirichlet character modulo $q$, we can show that $d_{\chi}(n)$ is non-negative for each $n$ by similar arguments given in Remark \ref{remark1}. From \eqref{defi1}, it can be easily seen that $d_{\chi}(n) \geq 1$ whenever $n$ is a perfect square. As $K_0(x)$ tends to $+\infty$ as $x$ decreases to $0$, the left-hand side of \eqref{Kosh_Bruce_Analog} approaches $+\infty$ as $x$ decreases to $0$. Now the infinite series in the right-hand side of \eqref{Kosh_Bruce_Analog} decreases rapidly as $x$ decreases to $0$. 
Therefore, we arrive at the conclusion that $\frac{2}{a^2x} L(1,\chi)$ tends to $+\infty$ as $x$ decreases to $0$ which proves the strict positivity of $L(1, \chi)$.
\end{remark}
			\subsection{Identities involving two characters.}
			In this subsection, we provide the identities corresponding to 
 $ \sigma_{k,\chi_1,\chi _2}(n)=\sum_{d/n}d^k \chi_1(d)\chi _2(n/d)$, where $\chi_1$ and $\chi_2$ are Dirichlet characters modulo $ p$ and $q$, respectively.
			\begin{theorem}\label{evenoddthm1}
				Let $k \geq 1$ be an odd integer. Let $\chi_1$ and $\chi_2$ be primitive characters modulo $ p$ and $q$, respectively, such that either both are non-principal even characters or both are odd characters. Then, for any  $\Re{(\nu)}>0$,   
   \begin{align*}
   \sum_{n=1}^\infty \sigma_{k,\chi_1,\chi_2}(n)n^{\frac{\nu}{2}}K_\nu(a\sqrt{nx})=&\frac{(-1)^{\frac{k+1}{2}}(aq)^\nu p^{\nu+k}\ x^{\frac{\nu}{2}}}{2^{3\nu+k+2}\pi^{2\nu+k+1}}   \tau(\chi_1)\tau(\chi_2)\Gamma(\nu+k+1)\sum_{n=1}^\infty \frac{\sigma_{k,\Bar{\chi_2},\Bar{\chi_1}}(n)} {\left( n+\frac{a^2pqx}{16\pi^2} \right)^{\nu+k+1} }.
\end{align*} 
			\end{theorem}
			 The result corresponding to $\nu=0$ is as follows \begin{theorem}\label{evenoddthm1nu=0} 
				Let $k \geq 1$ be an odd integer. Assume that $\chi_1$ and $\chi_2$ are primitive characters modulo $ p$ and $q$, respectively, such that either both are non-principal even characters or both are odd characters, then
				\begin{align*}
				 \sum_{n=1}^\infty \sigma_{k,\chi_1,\chi_2}(n) K_0(a\sqrt{nx})=  \frac{1}{2}c_{k,\chi_1,\chi_2}  +\frac{ (-1)^{ \frac{k-1}{2}}k! p^k }{2(2 \pi)^{k+1}} \tau(\chi_1) \tau(\chi_2) \sum_{n=1}^\infty \sigma_{k,\Bar{ \chi  }_2,\Bar{ \chi  }_1}(n)\left(\frac{1}{n^{k+1}}-\frac{1}{(n+ \frac{a^2pqx}{16\pi^2} )^{k+1} } \right),  
    \end{align*} 
    where $c_{k,\chi_1,\chi_2}$ is a constant defined as
    \begin{align}\label{cKK}
      c_{k,\chi_1,\chi_2}=\begin{cases}
          L(-k,\chi_1)L^\prime(0,\chi_2),     &\ \ \text{if both } \chi_1  \text{ and } \chi_2  \text{ are even},\\
           L^\prime(-k,\chi_1)L (0,\chi_2),     &\ \ \text{if both } \chi_1  \text{ and } \chi_2  \text{ are odd}.
      \end{cases}  
    \end{align}
     \end{theorem}
    \vspace{.1cm}
			  Setting $\chi_1=\chi_2=\chi$ and observing $\sigma_{k,\chi,\chi}(n)=\chi(n)\sum_{d/n}d^k=\chi(n)\sigma_k (n)$ in Theorems \ref{evenoddthm1} and \ref{evenoddthm1nu=0}, we obtain the following interesting identities.
			\begin{corollary}\label{last1 cor}
				Let $k\geq 1$ be an odd integer and $\chi$ be a non-principal primitive character modulo $q$. Then, for any  $\Re{(\nu)}>0$, 
				  \begin{align*}
  \sum_{n=1}^\infty  \sigma_{k}(n)\chi(n)n^{\frac{\nu}{2}}K_\nu(a\sqrt{nx})=&\frac{(-1)^{\frac{k+1}{2}}   a^\nu q^{2\nu+k}\ x^{\frac{\nu}{2}}}{2^{3\nu+k+2} \pi^{2\nu+k+1}}   \tau^2(\chi) \Gamma(\nu+k+1)\sum_{n=1}^\infty  \frac{\sigma_{k}(n) \ \Bar{\chi}(n) \ }{\left( n+\frac{a^2q^2x}{16\pi^2} \right)^{\nu+k+1} } .
\end{align*} 
			\end{corollary}
\begin{corollary}\label{last2 cor}
  Let $k\geq 1$ be an odd integer and $\chi$ be a non-principal primitive character modulo $q$. For $\nu=0$, we have 
\begin{align*}
				 \sum_{n=1}^\infty  \sigma_{k}(n)\chi(n) K_0(a\sqrt{nx})=& \frac{1}{2}c_{k,\chi,\chi}+\frac{ (-1)^{ \frac{k-1}{2}}k! q^k }{2(2 \pi)^{k+1}} \tau^2(\chi) \sum_{n=1}^\infty  \sigma_{k}(n) \ \Bar{\chi}(n) \ \left(\frac{1}{n^{k+1}}-\frac{1}{(n+ \frac{a^2q^2x}{16\pi^2} )^{k+1} } \right),  
	\end{align*} 
 where $c_{k,\chi,\chi}$ is defined in \eqref{cKK}.
\end{corollary}
		  The results corresponding to $\nu=0$ and $k=0 $ are as follows
     \begin{theorem}\label{EVENEVENNU=0k=0}   Let $\chi_1$ and $\chi_2$ be non-principal even primitive characters modulo $ p$  and $q$, respectively. Then    \begin{align*}
   \sum_{n=1}^\infty d_{\chi_1,\chi_2}(n)  K_0(a\sqrt{nx})=   \frac{a^2pq\ x}{32 \pi^4} \tau( \chi _1)\tau(\chi_2) \sum_{n=1}^\infty d_{\Bar{\chi}_1,\Bar{\chi}_2}(n)
      \frac{\log \left(\frac{16\pi^2 n}{a^2pqx} \right) }{n^2-\left(\frac{a^2pqx}{16\pi^2 }  \right)^2},
\end{align*} 
provided $\frac{a^2 pq x}{16 \pi^2} \notin \mathbb{Z}_+$.
     \end{theorem}
 \begin{theorem}\label{ODDODDnu0k0}
  Let $\chi_1$ and $\chi_2$ be odd primitive characters modulo $ p$ and $q$, respectively. Then  we have 
    \begin{align*}
				 	\sum_{n=1}^\infty d_{\chi_1,\chi_2}(n)  K_0(a\sqrt{nx})=& \frac{1}{2}L (0,\chi_1)L (0,\chi_2) \left( -2\gamma+  \log  \left( \frac{4}{a^2x}  \right)    +\frac{ L ^\prime(0,\chi_1)}{ L (0,\chi_1)}+\frac{L^\prime (0,\chi_2) }{L (0,\chi_2) } \right)    \\
        & +\frac{a^4p^2q^2 }{512\pi^4} \ x^2 \tau(\chi_1)\tau(\chi_2)\sum_{n=1}^\infty  \frac{d_{\Bar{\chi}_1,\Bar{\chi}_2}(n) \log \left(\frac{a^2pqx}{16\pi^2 n}\right)}{n\left( n^2-(\frac{a^2pqx}{16\pi^2} )^2\right) },
       \end{align*}
       provided $\frac{a^2 pq x}{16 \pi^2} \notin \mathbb{Z}_+$.
     \end{theorem}
			\begin{theorem}\label{evenoddthm2}
				 Let $k$ be an even, non-negative integer. Assume that $\chi_1$ and $\chi_2$ are primitive characters modulo $p$ and $q$, respectively, such that one is a non-principal even character and the other is an odd character.  Then, for any $\Re{(\nu)}>0$, 
				 \begin{align*}
	\sum_{n=1}^\infty \sigma_{k,\chi_1,\chi_2}(n)n^{\frac{\nu}{2}}K_\nu(a\sqrt{nx})=&\frac{(-1)^{\frac{k}{2}}\ (aq)^{\nu} p^{\nu+k}\ x^{\frac{\nu}{2}}}{i2^{3\nu+k+2}\pi^{2\nu+k+1}}    \tau(\chi_1)\tau(\chi_2) 
					\Gamma(\nu+k+1)\sum_{n=1}^\infty \frac{\sigma_{k,\Bar{\chi_2},\Bar{\chi_1}}(n)}{\left( n+\frac{a^2pqx}{16\pi^2} \right)^{\nu+k+1} } .
\end{align*} 
			\end{theorem}
 The result corresponding to $\nu=0$ is as follows 
    \begin{theorem}\label{18}
       Let $k$ be an even, non-negative integer. If $\chi_1$ and $\chi_2$ are primitive characters modulo $p$ and $q$, respectively, such that one is a non-principal even character and the other is an odd character, then
\begin{align*}
     \sum_{n=1}^\infty\sigma_{k,\chi_1,\chi_2}(n) K_0(a\sqrt{nx})=\frac{1}{2}e_{k,\chi_1,\chi_2}+(-1)^{\frac{k}{2}}\frac{ik!p^k}{2 ({2\pi })^{k+1} }  \tau(\chi_1)   \tau(\chi_2)   \sum_{n=1}^\infty  \sigma_{k,\Bar{ \chi  }_2,\Bar{ \chi  }_1}(n)   \left(\frac{1}{n^{k+1}}-\frac{1}{(n+ \frac{a^2pqx}{16\pi^2} )^{k+1} } \right), 
\end{align*}
where  
\begin{align}
e_{k,\chi_1,\chi_2}=\begin{cases}\label{e_kk}
          L(-k,\chi_1)L^\prime(0,\chi_2),   \ \ \text{if} \ \chi_1  \text{ is odd and } \chi_2  \text{ is even},\\
           L^\prime(-k,\chi_1)L (0,\chi_2),  \ \ \text{if } \ \chi_1  \text{ is even and }\chi_2  \text{ is odd}.
      \end{cases}  
    \end{align}
        \end{theorem}
		 \section{Cohen-Type Identities}\label{cohen identities...}
  This section deals with $z=-\nu$ with $\nu \notin \mathbb{Z}$ such that $\Re{(\nu)}\geq0 $. We will assume that $x$ is a strictly positive real number. 
  \subsection{Identities involving even characters and specializations.  }  In this subsection, we present the identities associated with $\sigma_{-\nu,\bar{\chi}}(n)$ and $\bar{\sigma}_{-\nu, \bar{\chi}}(n)$ when $\chi$ is a non-principal even primitive character. 
\begin{theorem}\label{evencohen}
   Let $\nu \notin \mathbb{Z}$ such that $\Re{(\nu)}\geq0 $. Let $\chi$ be a non-principal even primitive character modulo $ q$. If $N$ is any integer such that  $ N\geq \lfloor\frac{\Re{(\nu)}+1}{2}\rfloor$, then
\begin{align}\label{Cohen1}
	&8\pi x^{\nu/2}	\sum_{n=1}^{\infty} \sigma_{-\nu, \bar{\chi}}(n) n^{\nu/2}K_{\nu}(4\pi\sqrt{nx})  =-\frac{ \Gamma(\nu) L(\nu, \bar{\chi})}{(2\pi)^{\nu-1} }  +  \frac{2\Gamma(1+\nu) L(1+\nu, \bar{\chi})}{(2\pi)^{\nu+1} }x^{  -1} \nonumber \\
  & \ \ \ \ \ \   + \frac{2q^{1-\nu}  }{\tau(\chi) \sin \left(\frac{\pi \nu}{2}\right)}  
  \left\{\sum_{j=1}^{N} \zeta(2j)\ L(2j-\nu, \chi)(qx)^{2j-1} 
 + (qx)^{2N+1}\sum_{n=1}^{\infty}\bar{\sigma}_{-\nu, \chi}(n) \left(\frac{  n^{\nu-2N}-(qx)^{\nu-2N} }{ n^2-(qx)^2} \right)
    \right\},
			\end{align}
    provided 
    $qx\notin \mathbb{Z}_+$.
  \end{theorem}
  The specialization of the above theorem to $\nu=1/2$ is as follows
	 \begin{corollary}\label{Corcohen1} We have
    \begin{align*}
 2\pi \sum_{n=1}^\infty& \sigma_{-\frac{1}{2},\Bar{\chi}}(n)e^{-4\pi\sqrt{nx}} =-\pi L(1/2,\Bar{\chi})+\frac{1}{4\pi }L(3/2,\Bar{\chi}) x^{-1}
 +\frac{2q^{3/2}}{\tau(\chi)}x\sum_{n=1}^\infty\bar{\sigma}_{-\frac{1}{2}, \chi}(n) \frac{ 1}{ (n+qx)(\sqrt{n}+\sqrt{qx})} .
    \end{align*}
     \end{corollary}
\begin{theorem}\label{barevencohen}
  Let $\nu \notin \mathbb{Z}$ such that $\Re{(\nu)}\geq 0 .$ 
  Let $\chi$ be a non-principal even primitive character modulo $ q.$  If $N$ is any integer such that  $ N\geq \lfloor\frac{\Re{(\nu)}+1}{2}\rfloor$, then
 \begin{align*}
			&8\pi x^{\nu/2}	\sum_{n=1}^{\infty} \bar{\sigma}_{-\nu, \bar{\chi}}(n) n^{\nu/2} K_{\nu}(4\pi\sqrt{nx})   
   =     
   \frac{  q }{  \tau(\chi)}     \left\{
     \frac{ L(\nu,\chi)}{ \sin \left(\frac{\pi \nu}{2}\right)}    (qx)^{\nu-1} 
    - \frac{\pi L(1+\nu,\chi)}{  \cos \left(\frac{\pi \nu}{2}\right)}(qx)^{\nu}
     \right.\notag\\&\left.\ \ 
      \ \ \ \ \ +\frac{2}{ \sin \left(\frac{\pi \nu}{2}\right)}\sum_{j=1}^N {\zeta(2j-\nu)L(2j ,\chi)(qx)^{2j-1}}  
				+ \frac{2}{  \sin \left(\frac{\pi \nu}{2}\right)}(qx)^{2N+1}\sum_{n=1}^{\infty} {\sigma}_{-\nu, \chi}(n) \left(\frac{ n^{\nu-2N}-(qx)^{\nu-2N} }{n^2-(qx)^2}\right)
    \right\},
		\end{align*} provided $qx\notin \mathbb{Z}_+$.
  \end{theorem}
    The result corresponding to  $\nu=1/2$ is as follows
  \begin{corollary}\label{corcohen3} We have
    \begin{align*}
 2\pi \sum_{n=1}^\infty \bar{\sigma}_{-\frac{1}{2},\Bar{\chi}}(n)e^{-4\pi\sqrt{nx}} =& \frac{q^{1/2}}{\tau(\chi)} L(1/2,\chi)x^{-\frac{1}{2}}
   -\frac{\pi q^{3/2}}{\tau(\chi)}L(3/2,\chi)x^{\frac{1}{2}}          +\frac{2q^{2}}{\tau(\chi)}x\sum_{n=1}^\infty   \frac{ {\sigma}_{-\frac{1}{2}, \chi}(n)}{ (n+qx)(\sqrt{n}+\sqrt{qx})} .
    \end{align*}
    \end{corollary}
   \vspace{.001cm}
\subsection{Identities involving odd characters and specializations.}  In this subsection, we state the identities  associated with $\sigma_{-\nu,\bar{\chi}}(n)$ and $\bar{\sigma}_{-\nu, \bar{\chi}}(n)$  when $\chi$ is an odd primitive character. 
\begin{theorem}\label{oddcohen}
Let  $\nu \notin \mathbb{Z}$ such that $\Re{(\nu)}\geq 0 .$ Let $\chi$ be an odd primitive character modulo $ q.$  If $N$ is any integer such that  $ N\geq \lfloor\frac{\Re{(\nu)}+1}{2}\rfloor$, then
  \begin{align*}
  & 8\pi x^{\nu/2} \sum_{n=1}^{\infty} \sigma_{-\nu, \bar{\chi}}(n) n^{\nu/2} K_{\nu}(4\pi \sqrt{nx})
    =-\frac{ \Gamma(\nu) L(\nu, \bar{\chi})}{(2\pi)^{\nu-1} }  +  \frac{2\Gamma(1+\nu) L(1+\nu, \bar{\chi})}{(2\pi)^{\nu+1} }x^{-1} 
     + \frac{2iq^{1-\nu}  }{\tau(\chi) \cos \left(\frac{\pi \nu}{2}\right)} \\
     &\times\left\{  {\zeta(\nu+1)L(1,{\chi})(qx)^{\nu}}  
     -  \sum_{j=1}^{N} \zeta(2j)\ L(2j-\nu, \chi)(qx)^{2j-1} 
     - (qx)^{2N+1}\sum_{n=1}^{\infty} \frac{\bar{\sigma}_{-\nu, \chi}(n)}{n}    \left(\frac{  n^{\nu+1-2N}-(qx)^{\nu+1-2N} }{   n^2-(qx)^2}\right)
    \right\}, 
    	 \end{align*}
     provided   $qx\notin \mathbb{Z}_+$.
 \end{theorem}
  Setting $\nu=1/2$  in the above theorem, we obtain the following
  \begin{corollary} We have
    \begin{align*}
 2\pi \sum_{n=1}^\infty \sigma_{-\frac{1}{2},\Bar{\chi}}(n)e^{-4\pi\sqrt{nx}} =&-\pi L(1/2,\Bar{\chi})+\frac{1}{4\pi }L(3/2,\Bar{\chi}) x^{-1} +\frac{2i q}{\tau(\chi)}\zeta( 3/2)L(1,\chi)x^{1/2}\\
& -\frac{2i q^{3/2}}{\tau(\chi)}x\sum_{n=1}^\infty\bar{\sigma}_{-\frac{1}{2}, \chi}(n) \frac{(n+\sqrt{nqx}+qx )}{ n(n+qx)(n^{\frac{1}{2}}+(qx)^{\frac{1}{2}})} .
    \end{align*}
    \end{corollary}
 \begin{theorem}\label{baroddcohen}
Let  $\nu \notin \mathbb{Z}$ such that $\Re{(\nu)}\geq 0.$ Let $\chi$ be an odd primitive character modulo $ q.$  If $N$ is any integer such that  $ N\geq \lfloor\frac{\Re{(\nu)}+1}{2}\rfloor$, then
\begin{align*}
			&8\pi x^{\nu/2}	\sum_{n=1}^{\infty} \bar{\sigma}_{-\nu, \bar{\chi}}(n) n^{\nu/2} K_{\nu}(4\pi\sqrt{nx}) 
   = \frac{2\Gamma(\nu) \zeta(\nu)  L(0, \bar{\chi})}{(2\pi)^{\nu-1 }}    
    +\frac{ i q }{  \tau(\chi)}     \left\{  \frac{ L(\nu,\chi)}{  \cos \left(\frac{\pi \nu}{2}\right)}    (qx)^{\nu-1} + \frac{\pi L(1+\nu,\chi)}{  \sin\left(\frac{\pi \nu}{2}\right)}(qx)^{ \nu}
     \right.\notag\\&\left.\ \
     +\frac{2}{  \cos\left(\frac{\pi \nu}{2}\right)}\sum_{j=1}^{N-1} {\zeta(2j+1-\nu)L(2j+1 ,\chi)(qx)^{2j}} 
      + \frac{2  }{   \cos\left(\frac{\pi \nu}{2}\right)   }(qx)^{2N}\sum_{n=1}^{\infty} {\sigma}_{-\nu, \chi}(n) \left(\frac{n^{\nu+1-2N}-(qx)^{\nu+1-2N}}{n^2-(qx)^2}  \right) \right\},
\end{align*}
    provided $qx\notin \mathbb{Z}_+$.
 \end{theorem}
  The result corresponding to  $\nu=1/2$ is as follows
	  \begin{corollary} We have
     \begin{align*}
 2\pi \sum_{n=1}^\infty \bar{\sigma}_{-\frac{1}{2},\Bar{\chi}}(n)e^{-4\pi\sqrt{nx}} 
 =& 2\pi \zeta(1/2)L(0,\Bar{\chi})  +\frac{i q^{1/2}}{\tau(\chi)} L(1/2,\chi)x^{-\frac{1}{2}}  +\frac{\pi i q^{3/2}}{\tau(\chi)}L(3/2,\chi)x^{\frac{1}{2}}  \nonumber\\
 & +\frac{2i q }{\tau(\chi)}\sum_{n=1}^\infty {\sigma}_{-\frac{1}{2}, \chi}(n) \frac{ (n+\sqrt{nqx}+qx )}{ (n+qx)(\sqrt{n}+\sqrt{qx})}.
    \end{align*}
    \end{corollary}
   \vspace{.001cm}
			\subsection{Identities involving two characters and specializations.}
			Here we state the identities corresponding to $ \sigma_{-\nu,\chi_1,\chi _2}(n)=\sum_{d/n}d^{-\nu} \chi_1(d)\chi _2(n/d)$, where $\chi_1$ and $\chi_2$ are the Dirichlet characters modulo $p$ and $ q$, respectively.
 \begin{theorem}\label{cohenee}
Let $\nu \notin \mathbb{Z}$ such that $\Re{(\nu)}\geq0$. Both $\chi_1$ and $\chi_2$  are non-principal even primitive characters modulo  $p$ and $ q$, respectively.  If $N$ is any integer such that  $ N\geq \lfloor\frac{\Re{(\nu)}+1}{2}\rfloor$, then
\begin{align*}
	&	8\pi x^{\frac{\nu}{2}}\sum_{n=1}^{\infty}\sigma_{-\nu, \bar{\chi_1},\bar{\chi_2}}(n) n^{\nu/2} K_{\nu}(4\pi\sqrt{nx})   =\frac{ 2 p^{1-\nu}q }{\tau(\chi_1)\tau(\chi_2) \sin\left(\frac{\pi \nu}{2}\right)}\left\{\sum_{j=1}^{N} L(2j,\chi_2)\ L(2j-\nu, \chi_1)(pqx)^{2j-1} 
  \right.\notag\\&\left. 
  \ \ \ \ \ \ \ \ \ \   \ \   \ \ \ \ \ \ \ \ \ \   \ \ \ \ \ \ \ \ \ \ \ \ \ \ \ \ \ \ \ +{(pqx)^{2N+1} }\sum_{n=1}^\infty{\sigma}_{-\nu, \chi_2,\chi_1}(n) {  }\left(  \frac{n^{\nu-2N}-(pqx)^{\nu-2N}}{n^2-(pqx)^2}  \right)
    \right\},
			\end{align*}
   provided $pqx\notin \mathbb{Z}_+$.
		      \end{theorem}
Setting $\chi_1=\chi_2=\chi$ in the above theorem, we get the following
\begin{corollary}\label{coheneecor}
    Let $\nu \notin \mathbb{Z}$ such that $\Re{(\nu)}\geq0 $. Let $\chi$  be a non-principal even primitive character modulo $q$.  If $N$ is any integer such that  $ N\geq \lfloor\frac{\Re{(\nu)}+1}{2}\rfloor$, then
\begin{align*}
		&8\pi x^{\frac{\nu}{2}}\sum_{n=1}^{\infty} \sigma_{-\nu }(n) \Bar{\chi}(n) n^{\nu/2} K_{\nu}(4\pi\sqrt{nx})   =\frac{  2q^{2-\nu} }{\tau^2(\chi) \sin \left(\frac{\pi \nu}{2}\right)}\left\{ \sum_{j=1}^{N} L(2j,\chi)\ L(2j-\nu, \chi)(q^2x)^{2j-1}  
   \right.\notag\\&\left. 
  \ \ \ \ \ \ \ \ \ \   \ \   \ \ \ \   \ \ \ \ \ \ \ \ \ \   \ \   \ \ \ \ \ \ \ \ \ \   \ \ \ \ \ \ \ \ \ 
  + {(q^2x)^{2N+1} }\sum_{n=1}^\infty\sigma_{-\nu }(n)  {\chi}(n)  \left(  \frac{n^{\nu-2N}-(q^2x)^{\nu-2N}}{n^2-(q^2x)^2}  \right)
    \right\},
			\end{align*}
   provided $q^2x\notin \mathbb{Z}_+$.
		  \end{corollary}
	\begin{theorem}\label{cohenoo}
Let $\nu \notin \mathbb{Z}$ such that $\Re{(\nu)}\geq 0$. Both $\chi_1$ and $\chi_2$ are odd primitive characters modulo $p$ and $q$, respectively. If $N$ is any integer such that  $ N\geq \lfloor\frac{\Re{(\nu)}+1}{2}\rfloor$, then
  \begin{align*}
		&8 \pi x^{\nu/2}\sum_{n=1}^{\infty}\sigma_{-\nu, \bar{\chi_1},\bar{\chi_2}}(n) n^{\nu/2} K_{\nu}(4\pi\sqrt{nx}) =
 \Gamma(\nu)L(\nu, \bar{\chi_1})L(0, \bar{\chi_2})
        \frac{2 }{(2\pi)^{\nu-1}} \\
       & -\frac{ 2 p^{1-\nu}q}{  \tau(\chi_1)\tau(\chi_2)\sin \left(\frac{\pi \nu}{2}\right)}     \left\{ - {L(\nu+1,\chi_2)L(1,\chi_1)(pqx)^{\nu}}   
      + \sum_{j=1}^{N-1} L(2j+1,\chi_2)\ L(2j+1-\nu, \chi_1)(pqx)^{2j}
     \right.\notag\\&\left. 
   \ \   \ \ \ \   \ \   \ \ \ \   \ \ \ \ \ \ \ \ \ \ \ \ \ \   \ \   \ \ \ \ \ \ \ \   \ \   \ \ \ \ \ \
      +  (pqx)^{2N} \sum_{n=1}^{\infty}\frac{{\sigma}_{-\nu, \chi_2,\chi_1}(n)}{n}    \left( \frac{n^{\nu-2N+2}-(pqx)^{\nu-2N+2}}{n^2-(pqx)^2  } \right)
   \right\},
			\end{align*}
   provided $pqx\notin \mathbb{Z}_+$.
 \end{theorem}
 Taking $\chi_1=\chi_2=\chi$ in the above theorem, we get the following
 \begin{corollary}\label{cohenoocor}
      Let $\nu \notin \mathbb{Z}$ such that $\Re{(\nu)}\geq 0$. Let $\chi$  be an odd primitive character modulo $p$.  If $N$ is any integer such that  $ N\geq \lfloor\frac{\Re{(\nu)}+1}{2}\rfloor$, then
\begin{align*} 
		&8\pi x^{\frac{\nu}{2}}\sum_{n=1}^{\infty} \sigma_{-\nu }(n) \Bar{\chi}(n) n^{\nu/2} K_{\nu}(4\pi\sqrt{nx})   =
 \Gamma(\nu)L(\nu, \bar{\chi})L(0, \bar{\chi})
        \frac{2 }{(2\pi)^{\nu-1} } \notag \\  &-\frac{ 2 p^{2-\nu}}{  \tau^2(\chi) \sin \left(\frac{\pi \nu}{2}\right)}  
     \left\{ - {L(\nu+1,\chi_2)L(1,\chi_1)(p^2x)^{\nu}}   
      + \sum_{j=1}^{N-1} L(2j+1,\chi)\ L(2j+1-\nu, \chi)(p^2x)^{2j}
      \right.\notag\\&\left.\ \
      \ \ \ \ \ \ \ \ \  \ \ \ \ \ \ \ \ \ \ \ \ \ \ \ \  \ \ \ \ \ \ \ \ \ \ \ \ \ \ \ \  +  (p^2x)^{2N} \sum_{n=1}^{\infty}\frac{\sigma_{-\nu }(n)  {\chi}(n)}{n}    \left( \frac{n^{\nu-2N+2}-(p^2x)^{\nu-2N+2}}{n^2-(p^2x)^2  } \right)
   \right\},
			\end{align*}
   provided $p^2x\notin \mathbb{Z}_+$.
  \end{corollary}
\begin{theorem}\label{coheneo}
 Let $\nu \notin \mathbb{Z}$ such that $\Re{(\nu)}\geq0$. Let  $\chi_1$ be a non-principal even primitive character modulo  $ p$ and 
$\chi_2$ be an odd primitive character modulo  $ q$. If $N$ is any integer such that  $ N\geq \lfloor\frac{\Re{(\nu)}+1}{2}\rfloor$, then
 \begin{align*}
	&8 \pi x^{\nu/2}	\sum_{n=1}^{\infty} \sigma_{-\nu, \bar{\chi_1},\bar{\chi_2}}(n) n^{\nu/2} K_{\nu}(4\pi\sqrt{nx}) =\frac{ 2 }{(2\pi)^{\nu-1}}
 \Gamma(\nu)L(\nu, \bar{\chi_1})L(0, \bar{\chi_2})
          +\frac{2i  p^{1-\nu}q}{\tau(\chi_1)\tau(\chi_2) \cos\left(\frac{\pi \nu}{2}\right)   }  \\
          &\times\left\{   \sum_{j=1}^{N-1} L(2j+1,\chi_2)  L(2j+1-\nu, \chi_1)(pqx)^{2j} 
   + (pqx)^{2N}\sum_{n=1}^{\infty} {\sigma}_{-\nu, \chi_2,\chi_1}(n) \left(\frac{n^{\nu-2N+1}-(pqx)^{\nu-2N+1}}{n^2-(pqx)^2}  \right)
 \right\},
			\end{align*}
   provided $pqx\notin \mathbb{Z}_+$.
\end{theorem}
\begin{theorem}\label{cohenoe}
  Let $\nu \notin \mathbb{Z}$ such that $\Re{(\nu)}\geq 0$. Let  $\chi_1$ be an  odd   primitive character modulo  $ p$ and 
$\chi_2$ be  a non-principal even  primitive character modulo $ q$.  If $N$ is any integer such that  $ N\geq \lfloor\frac{\Re{(\nu)}+1}{2}\rfloor$, then
\begin{align*}
				&8\pi x^{\nu/2}\sum_{n=1}^{\infty} \sigma_{-\nu, \bar{\chi_1},\bar{\chi_2}}(n) n^{\nu/2} K_{\nu}(4\pi\sqrt{nx})  
    =\frac{2i  p^{1-\nu}q}{  \tau(\chi_1)\tau(\chi_2) \cos \left(\frac{\pi \nu}{2}\right)}     \left\{  {L(\nu+1,\chi_2)L(1,\chi_1)(pqx)^{\nu}} 
     \right.\notag\\&\left.\ \
    - \sum_{j=1}^{N} L(2j,\chi_2)\ L(2j-\nu, \chi_1)(pqx)^{2j-1}
      -  (pqx)^{2N+1}\sum_{n=1}^{\infty}\frac{{\sigma}_{-\nu, \chi_2,\chi_1}(n) }{n}  \left( \frac{n^{\nu-2N+1}-(pqx)^{\nu-2N+1}}{n^2-(pqx)^2} \right)
	 \right\},
		\end{align*} 
  provided $pqx\notin \mathbb{Z}_+$.
 \end{theorem}
 \section{Connection with Vorono\"i summation formula}\label{voronoi identities...}
  In this section, we offer Vorono\"i-type summation formulas for $ \sigma_{z,\chi}(n),\bar{\sigma}_{z,\chi}(n)$ and  $\sigma_{z,\chi_1,\chi_2}(n)$  defined in \eqref{defimainpaper}.
  \subsection{Identities involving even characters. }
 \begin{theorem}\label{vore1}
Let $0< \alpha< \beta$ and $\alpha, \beta \notin \mathbb{Z}$. Let $f$ denote a function analytic inside a closed contour strictly containing $[\alpha, \beta ]$. Assume that   $\chi$  is a  non-principal even primitive character modulo $ q$. For  $0< \Re{(\nu)}<\frac{1}{2}$, we have   
        \begin{align*} 
           &\frac{   q^{1-\frac{\nu}{2}}  } {\tau(\chi )}    \sum_{\alpha<j <\beta}    {\bar{\sigma}_{-\nu, \chi }(j)} f(j)   =   \frac{ q^{1-\frac{\nu}{2}}}{\tau(\chi)} L(1-\nu,\chi) \int_\alpha^\beta    \frac{f(t) }{t^{\nu}}  \mathrm{d}t+2\pi  \sum_{n=1}^{\infty}\sigma_{-\nu, \bar{\chi }}(n) \ n^{\nu/2}  \int_{\alpha} ^{  \beta   }f(t)  (t)^{-\frac{\nu}{2}}  \notag\\ 
  &   \times     \left\{  \left(   \frac{2}{\pi} K_{\nu}\left(4\pi  \sqrt{\frac{nt}{q}}\ \right) - Y_{\nu}\left(4\pi  \sqrt{\frac{nt}{q}}\ \right)   \right) \cos \left(\frac{\pi \nu}{2}\right) 
    -  J_{\nu}\left(4\pi  \sqrt{\frac{nt}{q}}\ \right)  \sin \left(\frac{\pi \nu}{2}\right)  \right\}  dt.
           \end{align*}
 \end{theorem}
\begin{theorem}\label{vore2}
Let $0< \alpha< \beta$ and $\alpha, \beta \notin \mathbb{Z}$. Let $f$ denote a function analytic inside a closed contour strictly containing $[\alpha, \beta ]$. Assume that   $\chi$  is a  non-principal even primitive character modulo $ q$.  For  $0< \Re{(\nu)}<\frac{1}{2}$, we have  
\begin{align*}
&\frac{   q^{1+\frac{\nu}{2}}  } {\tau(\chi )}    \sum_{\alpha<j <\beta}    { {\sigma}_{-\nu, \chi }(j)} f(j)   =   \frac{ q^{1+\frac{\nu}{2}}}{\tau(\chi)} L(1+\nu,\chi) \int_\alpha^\beta     {f(t) } \mathrm{d}t
 + 2 \pi \sum_{n=1}^{\infty}\bar{\sigma}_{-\nu, \bar{\chi }}(n) \ n^{\nu/2} \int_{\alpha} ^{  \beta   }f(t) (t)^{-\frac{\nu}{2}}  \notag\\ 
  &   \times     \left\{  \left(   \frac{2}{\pi} K_{\nu}\left(4\pi  \sqrt{\frac{nt}{q}}\ \right) - Y_{\nu}\left(4\pi  \sqrt{\frac{nt}{q}}\ \right)   \right) \cos \left(\frac{\pi \nu}{2}\right) 
    -  J_{\nu}\left(4\pi  \sqrt{\frac{nt}{q}}\ \right)  \sin \left(\frac{\pi \nu}{2}\right)  \right\}  dt.
\end{align*}
\end{theorem}
 \vspace{.001cm}
\subsection{Identities involving odd characters.  }
\begin{theorem}\label{voro1}
 Let $0< \alpha< \beta$ and $\alpha, \beta \notin \mathbb{Z}$. Let $f$ denote a function analytic inside a closed contour strictly containing $[\alpha, \beta ]$. Assume that  $\chi$ is an odd primitive character modulo  $q$.   For  $0 < \Re{(\nu)}<\frac{1}{2}$, we have  
 \begin{align*}     
           & \frac{   q^{1-\frac{\nu}{2}}  } {\tau(\chi )}    \sum_{\alpha<j <\beta}    \frac{\bar{\sigma}_{-\nu, \chi }(j)}{j} f(j)  = \frac{ q^{1-\frac{\nu}{2}}}{\tau(\chi)} L(1-\nu,\chi) \int_\alpha^\beta    \frac{f(t) }{t^{\nu+1}}  \mathrm{d}t-2\pi i  \sum_{n=1}^{\infty}\sigma_{-\nu, \bar{\chi }}(n) \ n^{\nu/2}  \int_{\alpha} ^{  \beta   }f(t)  (t)^{-\frac{\nu}{2}-1}  \notag\\ 
  &   \times     \left\{  \left(   \frac{2}{\pi} K_{\nu}\left(4\pi  \sqrt{\frac{nt}{q}}\ \right) - Y_{\nu}\left(4\pi  \sqrt{\frac{nt}{q}}\ \right)   \right) \sin \left(\frac{\pi \nu}{2}\right)  
    + J_{\nu}\left(4\pi  \sqrt{\frac{nt}{q}}\ \right)  \cos \left(\frac{\pi \nu}{2}\right)   \right\}  dt.
           \end{align*}
    \end{theorem}
\begin{theorem}\label{voro2}
Let $0< \alpha< \beta$ and $\alpha, \beta \notin \mathbb{Z}$. Let $f$ denote a function analytic inside a closed contour strictly containing $[\alpha, \beta ]$. Assume that   $\chi$  is an odd primitive character modulo  $q$. For $0< \Re{(\nu)}<\frac{1}{2}$, we have  
\begin{align*}
&\frac{   q^{1+\frac{\nu}{2}}  } {\tau(\chi )}    \sum_{\alpha<j <\beta}    { {\sigma}_{-\nu, \chi }(j)} f(j)   =   \frac{ q^{1+\frac{\nu}{2}}}{\tau(\chi)} L(1+\nu,\chi) \int_\alpha^\beta     {f(t) } \mathrm{d}t
 + 2 \pi i \sum_{n=1}^{\infty}\bar{\sigma}_{-\nu, \bar{\chi }}(n) \ n^{\nu/2} \int_{\alpha} ^{  \beta   }f(t) (t)^{-\frac{\nu}{2}}  \notag\\ 
  &   \times     \left\{  \left(   \frac{2}{\pi} K_{\nu}\left(4\pi  \sqrt{\frac{nt}{q}}\ \right) + Y_{\nu}\left(4\pi  \sqrt{\frac{nt}{q}}\ \right)   \right)  \sin \left(\frac{\pi \nu}{2}\right)  
    -  J_{\nu}\left(4\pi  \sqrt{\frac{nt}{q}}\ \right)  \cos \left(\frac{\pi \nu}{2}\right)   \right\}  dt.
\end{align*}
\end{theorem}
\vspace{.001cm}
 \subsection{Identities involving two characters.}
			In this subsection, we state Vorono\"i-type summation formula associated with $ \sigma_{-\nu,\chi_1,\chi _2}(n)=\sum_{d/n}d^{-\nu} \chi_1(d)\chi _2(n/d)$, where $\chi_1$ and $\chi_2$ are  Dirichlet characters modulo  $ p$ and $ q,$ respectively.
\begin{theorem}\label{voree} 
Let $0< \alpha< \beta$ and $\alpha, \beta \notin \mathbb{Z}$. Let $f$ denote a function analytic inside a closed contour strictly containing $[\alpha, \beta ]$. Assume that $\chi_1$ and $\chi_2$ are non-principal even primitive characters modulo  $ p$ and $ q,$ respectively. For  $0< \Re{(\nu)}<\frac{1}{2}$, we have 
\begin{align*}     
          & \frac{  p^{1-{\frac{\nu}{2}}}q^{1+{\frac{\nu}{2}}} }{  \tau(\chi_1)\tau(\chi_2)}   \sum_{\alpha<j <\beta}    {\sigma}_{-\nu, \chi_2,\chi_1}(j)  f(j)   = 2\pi  \sum_{n=1}^{\infty}\sigma_{-\nu, \bar{\chi_1},\bar{\chi_2}}(n) \ n^{\nu/2} \int_{\alpha} ^{  \beta   }f(t)  (t)^{-\frac{\nu}{2}}  \notag\\ 
  &   \times     \left\{  \left(   \frac{2}{\pi} K_{\nu}\left(4\pi  \sqrt{\frac{nt}{pq}}\ \right) - Y_{\nu}\left(4\pi  \sqrt{\frac{nt}{pq}}\ \right)   \right) \cos \left(\frac{\pi \nu}{2}\right) 
    -  J_{\nu}\left(4\pi  \sqrt{\frac{nt}{pq}}\ \right)  \sin \left(\frac{\pi \nu}{2}\right)  \right\}  dt.
           \end{align*}
 \end{theorem}
Substituting $\chi_1=\chi_2=\chi$ in the above theorem, we get the following 
\begin{corollary}
 Let $0< \alpha< \beta$ and $\alpha, \beta \notin \mathbb{Z}$. Let $f$ denote a function analytic inside a closed contour strictly containing $[\alpha, \beta ]$. Assume that $\chi$  is a non-principal even primitive character modulo   $ q.$  For  $0< \Re{(\nu)}<\frac{1}{2}$, we have 
\begin{align*}     
          & \frac{   q^2 }{  \tau^2(\chi) }   \sum_{\alpha<j <\beta}    {\sigma}_{-\nu }(j) \chi(j) f(j)   = 2 \pi \sum_{n=1}^{\infty}\sigma_{-\nu}(n)\Bar{\chi}(j) \ n^{\nu/2} \int_{\alpha} ^{  \beta   }f(t)  (t)^{-\frac{\nu}{2}}  \notag\\ 
  &   \times     \left\{  \left(   \frac{2}{\pi} K_{\nu}\left(4\pi  \sqrt{\frac{nt}{q^2}}\ \right) - Y_{\nu}\left(4\pi  \sqrt{\frac{nt}{q^2}}\ \right)   \right) \cos \left(\frac{\pi \nu}{2}\right) 
    -  J_{\nu}\left(4\pi  \sqrt{\frac{nt}{q^2}}\ \right)  \sin \left(\frac{\pi \nu}{2}\right)  \right\}  dt.
           \end{align*}
\end{corollary}
 \begin{theorem}\label{voroo}
Let $0< \alpha< \beta$ and $\alpha, \beta \notin \mathbb{Z}$. Let $f$ denote a function analytic inside a closed contour strictly containing $[\alpha, \beta ]$. Assume that $\chi_1$ and $\chi_2$ are odd primitive characters modulo  $ p$ and $ q$, respectively. For  $0< \Re{(\nu)}<\frac{1}{2}$, we have     
        \begin{align*}     
          & \frac{  p^{1-{\frac{\nu}{2}}}q^{1+{\frac{\nu}{2}}} }{  \tau(\chi_1)\tau(\chi_2)}   \sum_{\alpha<j <\beta}   \frac{{\sigma}_{-\nu, \chi_2,\chi_1}(j)  f(j) }{j}   = -2 \pi \sum_{n=1}^{\infty}\sigma_{-\nu, \bar{\chi_1},\bar{\chi_2}}(n) \ n^{\nu/2} \int_{\alpha} ^{  \beta   }f(t)  (t)^{-\frac{\nu}{2}-1}  \notag\\ 
  &   \times     \left\{  \left(   \frac{2}{\pi} K_{\nu}\left(4\pi  \sqrt{\frac{nt}{pq}}\ \right) + Y_{\nu}\left(4\pi  \sqrt{\frac{nt}{pq}}\ \right)   \right) \cos \left(\frac{\pi \nu}{2}\right) 
    + J_{\nu}\left(4\pi  \sqrt{\frac{nt}{pq}}\ \right)  \sin \left(\frac{\pi \nu}{2}\right)  \right\}  dt.
           \end{align*}
\end{theorem}
 Substituting $\chi_1=\chi_2=\chi$ in the above theorem, we get the following
\begin{corollary}
 Let $0< \alpha< \beta$ and $\alpha, \beta \notin \mathbb{Z}$. Let $f$ denote a function analytic inside a closed contour strictly containing $[\alpha, \beta ]$. Assume that $\chi$  is an odd primitive character modulo  $ q.$ For  $0< \Re{(\nu)}<\frac{1}{2}$, we have
\begin{align*}     
          & \frac{   q^2}{  \tau^2(\chi )}   \sum_{\alpha<j <\beta}   \frac{ {\sigma}_{-\nu }(j) \chi(j) f(j) }{j}    = -2 \pi \sum_{n=1}^{\infty}\sigma_{-\nu}(n)\Bar{\chi}(j) \ n^{\nu/2} \int_{\alpha} ^{  \beta   }f(t)  (t)^{-\frac{\nu}{2}-1}  \notag\\ 
  &   \times     \left\{  \left(   \frac{2}{\pi} K_{\nu}\left(4\pi  \sqrt{\frac{nt}{q^2}}\ \right) + Y_{\nu}\left(4\pi  \sqrt{\frac{nt}{q^2}}\ \right)   \right) \cos \left(\frac{\pi \nu}{2}\right) 
    + J_{\nu}\left(4\pi  \sqrt{\frac{nt}{q^2}}\ \right)  \sin \left(\frac{\pi \nu}{2}\right)  \right\}  dt.
           \end{align*}
\end{corollary}
 \begin{theorem}\label{voreo}
Let $0< \alpha< \beta$ and $\alpha, \beta \notin \mathbb{Z}$. Let $f$ denote a function analytic inside a closed contour strictly containing $[\alpha, \beta ]$. Assume that   $\chi_1$ is a non-principal  even primitive character modulo $ p$ and 
$\chi_2$ is an odd primitive character modulo  $ q.$ For  $0< \Re{(\nu)}<\frac{1}{2}$, we have  
 \begin{align*}     
          & \frac{  p^{1-{\frac{\nu}{2}}}q^{1+{\frac{\nu}{2}}} }{  \tau(\chi_1)\tau(\chi_2)}   \sum_{\alpha<j <\beta}   {\sigma}_{-\nu, \chi_2,\chi_1}(j)  f(j)   = 2 \pi i \sum_{n=1}^{\infty}\sigma_{-\nu, \bar{\chi_1},\bar{\chi_2}}(n) \ n^{\nu/2} \int_{\alpha} ^{  \beta   }f(t)  (t)^{-\frac{\nu}{2} }  \notag\\ 
  &   \times     \left\{  \left(   \frac{2}{\pi} K_{\nu}\left(4\pi  \sqrt{\frac{nt}{pq}}\ \right) + Y_{\nu}\left(4\pi  \sqrt{\frac{nt}{pq}}\ \right)   \right)  \sin   \left(\frac{\pi \nu}{2}\right) 
    - J_{\nu}\left(4\pi  \sqrt{\frac{nt}{pq}}\ \right)  \cos \left(\frac{\pi \nu}{2}\right)  \right\}  dt.
           \end{align*}
    \end{theorem}
 \begin{theorem}\label{voroe}
 Let $0< \alpha< \beta$ and $\alpha, \beta \notin \mathbb{Z}$. Let $f$ denote a function analytic inside a closed contour strictly containing $[\alpha, \beta ]$. 
Assume that  $\chi_1$ is an odd primitive character modulo  $ p$ and $\chi_2$ is a non-principal even primitive character modulo  $ q.$  For  $0< \Re{(\nu)}<\frac{1}{2}$, we have  
 \begin{align*}     
          & \frac{  p^{1-{\frac{\nu}{2}}}q^{1+{\frac{\nu}{2}}} }{  \tau(\chi_1)\tau(\chi_2)}   \sum_{\alpha<j <\beta}    \frac{{\sigma}_{-\nu, \chi_2,\chi_1}(j)f(j)}{j}    =- 2 \pi i\sum_{n=1}^{\infty}\sigma_{-\nu, \bar{\chi_1},\bar{\chi_2}}(n) \ n^{\nu/2} \int_{\alpha} ^{  \beta   }f(t)  t^{-\frac{\nu}{2}-1}  \notag\\ 
  &   \times     \left\{  \left(   \frac{2}{\pi} K_{\nu}\left(4\pi  \sqrt{\frac{nt}{pq}}\ \right) - Y_{\nu}\left(4\pi  \sqrt{\frac{nt}{pq}}\ \right)   \right)\sin \left(\frac{\pi \nu}{2}\right)
    + J_{\nu}\left(4\pi  \sqrt{\frac{nt}{pq}}\ \right) \cos \left(\frac{\pi \nu}{2}\right)     \right\}  dt.
           \end{align*}
  \end{theorem} 
 \vspace{.1cm}
		\section{Preliminaries}\label{preliminary}
			We begin this section by recalling and proving some important results which will be used throughout the paper.
			\par The Mellin transform of a locally integrable function $f(x)$ on $(0, \infty)$ is defined by
			\begin{align}\label{Mellin}
				\mathcal{M}[f; s] = F(s) =\int_{0}^{\infty} f(t) \ t^{s-1} dt,
			\end{align}
			provided the integral converges. The basic properties of the Mellin transform follow
			immediately from those of the Laplace transform since these transforms are intimately connected.  The integral in \eqref{Mellin} defines the Mellin transform in a vertical strip  in the $s$ plane whose boundaries are determined by the analytic structure of $f(x)$ as $x\rightarrow 0+$ and $x \rightarrow +\infty$. If we assume that $f(x)$ satisfies the following growth condition
			\begin{align}\label{cond}
				f(x)=\begin{cases}
					& O( x^{-a-\varepsilon} ) \ \ \mbox{as} \ x\rightarrow 0+,\\
					& O( x^{-b+\varepsilon} ) \ \ \mbox{as} \ x\rightarrow +\infty,
				\end{cases}
			\end{align}
		where $\varepsilon > 0$ and $a<b$, then the integral \eqref{Mellin} converges absolutely in the strip $a < \Re(s) < b$ and defines an analytic function there in the strip.
			This strip is known as the strip of analyticity of $\mathcal{M}[f; s]$.
			Furthermore, the inversion formula for \eqref{Mellin} follows directly from the corresponding inversion formula for the bilateral Laplace transform. Thus,
			\begin{align} \label{inverse Mellin}
				f(x)=\frac{1}{2\pi i} \int_{c-i\infty}^{c+i \infty} x^{-s} \mathcal{M}[f; s] ds \  \ \ \  (a < c< b), 
			\end{align}
			which is valid at all points $x \geq 0$ where $f(x)$ is continuous.
For example, $\mathcal{M}[e^x; s]=\Gamma(s)$ for $\Re(s)>0$, and we have the corresponding Mellin's inversion formula 
\begin{align*}
   e^{-y}=  \frac{1}{2\pi i}\int_{(c)}\Gamma(s)y^{-s}\mathrm{d}s, 
\end{align*}
valid for $\Re{(y)}>0$. The functional relations for $\Gamma(s)$ are given by \cite[p.~73]{MR1790423}
				\begin{align}
					 \Gamma(s+1)=&s \Gamma(s), \ \ \ \ \ \ \ \ \Gamma(s)\Gamma\left(s+\frac{1}{2}\right)=2^{1-2s}\sqrt{\pi}\Gamma(2s),
     \label{relation of gamma}
			\\
        & \Gamma(s)\Gamma(1-s)=\frac{\pi}{\sin (\pi s
)}.\label{reflection formula for gamma add}
   \end{align}
The following lemma states the asymptotic behaviour of $\Gamma(s)$.
			\begin{lemma}\label{Gamma}
				\cite[p.~38]{MR0364103} In a vertical strip, for s=$\sigma+it$ with $a\leq \sigma \leq b$ and $|t|\geq 1$,
				\begin{align*}
					|\Gamma(s)|=(2\pi)^\frac{1}{2}|t|^{\sigma-\frac{1}{2}}\exp^{-\frac{1}{2}\pi |t|}\left( 1+O\left(\frac{1}{|t|}\right)\right).
				\end{align*}
			\end{lemma}
			In our investigation, we shall require the following results related to the Mellin transform of derivatives of a function. 
			{ \begin{lemma}\label{lem1}
					Let $n \in \mathbb{N}$. Assume that $\phi$ is $n$-times differentiable function and
					\begin{align}\label{M1}
						\mathcal{M}[\phi(t); s]=\int_{0}^{\infty} \phi(t) t^{s-1} dt=\Phi(s).
					\end{align}
					If $\phi$ satisfies \eqref{cond}, then     
					\begin{align}\label{M2}
						\mathcal{M}[\phi^{(n)}(t) t^n; s]=(-1)^n\frac{\Gamma(s+n)}{\Gamma(s)}\Phi(s),
					\end{align}
					where $s\in \{w \in \mathbb{C}; \ a< \Re(w) <b\}$, provided 
					\begin{align}\label{M3}
						\lim_{t \rightarrow 0, \infty} t^{s+n-j-1} \phi^{(n-j-1)}(t)=0 \ \ j=0, 1, \cdots, n-1.
					\end{align}
				\end{lemma}
			}
			\begin{proof}
			The proof relies on mathematical induction.	Using integration by parts, we have
				\begin{align*}
					\mathcal{M}[x\phi^{\prime}(x) ; s] &=\int_{0}^{\infty} \phi^{\prime}(t) \ t^{s} dt=\left[t^{s}\phi(t)\right]_{0}^{\infty}-s\int_{0}^{\infty} \phi(t) \ t^{s-1} dt.
				\end{align*}
				Noting $\phi(t)$ satisfies \eqref{cond}, we can claim that
				\begin{align*}
					\mathcal{M}[x\phi^{\prime}(x) ; s] =-s \Phi(s) \ \ \mbox{for}  \ a < \Re(s) < b. 
				\end{align*}
				Suppose the statement of the theorem is true for $n=N$ and $\phi$ is $N+1$-times differentiable function and satisfies \eqref{M3}. Then
				\begin{align*}
					\mathcal{M}[x^{N+1} \phi^{(N+1)}(x) ; s] &=\int_{0}^{\infty} t^{N+1}\phi^{(N+1) }(t) \ t^{s-1} dt\\
					&=\left[t^{s+N} \phi^{(N)}(t)\right]_{0}^{\infty}-(s+N)\int_{0}^{\infty} t^{N}\phi^{(N)}(t) \ t^{s-1} dt.
				\end{align*}
				As $\phi$  satisfies \eqref{M3}, so we have
				\begin{align*}
					\mathcal{M}[\phi^{(N+1)}(t) t^{N+1}; s]=-(s+N)\int_{0}^{\infty} t^{N}\phi^{(N)}(t) \ t^{s-1} dt=(-1)^{N+1} \frac{\Gamma(s+N+1)}{\Gamma(s)}\Phi(s),
				\end{align*}
		and this completes the proof.
	 \end{proof}
	 	\begin{lemma}\label{lem2} 
				\cite[p.~91, Formula (3.3.9)]{MR1854469} We have 
				\begin{align*}
					\mathcal{M}[ (1+x)^{-a}; s]=\frac{ \Gamma(s)\Gamma(a-s)}{\Gamma(a)},
				\end{align*}
				for $0<\Re(s)<\Re(a)$.
			\end{lemma}
			 As an immediate consequence of Lemma \ref{lem2} we get,
			\begin{lemma}\label{lem3}
				For any $n \in \mathbb{N}$, 
				\begin{align*}
					\mathcal{M}\left[ \frac{a(a+1) \cdots (a+n-1) x^{n} }{(1+x)^{a+n} }; s\right]=\frac{ \Gamma(s+n) \Gamma(a-s)}{ \Gamma(a)},
				\end{align*}
				whenever $0<\Re(s)<\Re(a)$. 
			\end{lemma}
			\begin{proof}
				By Lemma \ref{lem2}, we can write
				\begin{align*}
					\mathcal{M}[(1+x)^{-a} ; s]=\frac{ \Gamma(s)\Gamma(a-s)}{\Gamma(a)},
				\end{align*}
				for $0<\Re(s)<\Re(a)$. The function $\phi(t)= \frac{1}{(1+t)^{a}}$ for $t\geq 0$ is a continuous function and satisfies all the conditions of Lemma \ref{lem1}. Furthermore, $$\phi^{(n)}(t) = (-1)^n\frac{a(a+1) \cdots (a+n-1)  }{(1+t)^{(a+n)}}.$$  We have $$\Phi(s)=\frac{\Gamma(a-s)\Gamma(s)}{\Gamma(a)} \ \  \mbox{for} \  0<\Re(s)<\Re(a).$$ Hence by Lemma \ref{lem1},
				\begin{align*}
					\mathcal{M}\left[  \frac{a(a+1) \cdots (a+n-1) t^n }{(1+x)^{(a+1)}} ; s\right]=\frac{\Gamma(s+n)}{\Gamma(s)}\Phi(s)=\frac{\Gamma(s+n) \Gamma(a-s)}{\Gamma(a)},
				\end{align*}
				for $0<\Re(s)<\Re(a)$. 
			\end{proof}
			\begin{lemma}\label{lem4}
				Let $n\geq0$ be any integer and $t>0$ be any real number. Then
				\begin{align*}
				 	\frac{1} {2\pi i}\int_{(c)}  { \Gamma(s+n) \Gamma(a-s)} t^{-s}\mathrm{d}s= \frac{ \Gamma(a+n)     }{(1+t)^{a+n} }t^{n},
				\end{align*}
				for $0<c<\Re(a)$.
			\end{lemma}
			\begin{proof}
				We get our desired result by combining Lemmas \ref{lem2} and \ref{lem3} and applying Mellin's inversion formula.
			\end{proof}
		 \begin{lemma}\label{G9} 
 \cite[p.~346, Formula (20)]{MR0061695}  We have
			\begin{align*}
					\mathcal{M}\left[ \frac{\log t}{t-1}; s \right] = \frac{\pi^2}{\sin^2(\pi s)} ,
				\end{align*}	 
				for $0<\Re(s)<1$.  The integral is convergent in the sense of Cauchy's principal value.
			\end{lemma}
			  \begin{lemma}\label{sinsquare}
    We have
				\begin{align}\label{sinpi2}
     \mathcal{M}\left[\frac{4 \ log \ x }{x^2-1}\ ; s\right]= \frac{\pi^2}{   \sin^2\left( \frac{\pi s}{2} \right)},
				\end{align}
    for $0<\Re(s)<2$. The integral is convergent in the sense of Cauchy's principal value.
			\end{lemma}
			\begin{proof}
				This is a direct consequence of Lemma \ref{G9}.
			\end{proof}
		 Now, we  record a few important results related to the modified $K$-Bessel function $K_\nu(x)$ defined by \eqref{Kbessel}.
			\begin{lemma}\label{eq:bessel}
   \cite[p.~10, Lemma 3.3]{debika2023}
				Let $\nu \in \mathbf{C}$. For any $c> \max\{0,-\Re(\nu)\}$, we have
				\begin{align*}
					t^{\frac{\nu}{2}}K_{\nu}(a\sqrt{tx})=\frac{1}{2}\left( \frac{2}{a\sqrt{x}}\right)^\nu \frac{1}{2\pi i}
					\int_{(c)}\Gamma(s) \Gamma(s+\nu)\left( \frac{4}{a^2x}\right)^st^{-s} ds.   \end{align*}
			\end{lemma}
			\medskip
		 We first observe that the generating functions for $\sigma_{z, \chi}(n)$ and $\Bar{\sigma}_{z,\chi}(n)$ and $\sigma_{z, \chi_1, \chi_2}(n)$ defined in \eqref{defimainpaper} are the following
			\begin{align}
				\zeta(s)L(s-z, \chi)&=\sum_{m=1}^{\infty} \frac{1}{m^s} \sum_{d=1}^{\infty} \frac{d^z \chi(d)}{d^s} =  \sum_{n=1}^{\infty} \frac{\sigma_{z, \chi}(n)}{n^s}, \label{Lfz1} \\
				\zeta(s-z)L(s,\chi)&=\sum_{m=1}^{\infty} \frac{1}{m^{s-z}} \sum_{d=1}^{\infty} \frac{ \chi(d)}{d^s} =\sum_{n=1}^\infty \frac{\Bar{\sigma}_{z,\chi}(n)}{n^s},\label{Lfz2}\\
				L(s-z,\chi_1)L(s,\chi_2)&=\sum_{d=1}^{\infty} \frac{d^z \chi_1(d)}{d^s} \sum_{m=1}^{\infty} \frac{\chi_2(m)}{m^s} =\sum_{n=1}^\infty \frac{ \sigma_{z,\chi_1,\chi_2}(n)}{n^s},\label{Lfz3}
			\end{align} 
			for $\Re(s)> max(\Re(z)+1, 1)$, where $\zeta(s)$  denotes the the Riemann zeta function and $L(s,\chi)$ denotes the Dirichlet $L$-function defined by \eqref{Lfunctiondefi} for $\Re(s)>1$.
			We recall that the functional equation of $\zeta(s)$ \cite[p.~234]{MR3363366}
 \begin{align}\label{lq2}
 \zeta(s) = 2^s \pi^{s-1}\sin \left(\frac{\pi s}{2}\right)   \Gamma(1-s)\zeta(1-s).
   \end{align}
   Replacing $s$ by $1-s$ in \eqref{lq2}, we obtain
			\begin{align}\label{1st_use}
				\Gamma(s)  \zeta(s) & = \frac{\pi^{s} \zeta(1-s)}{2^{1-s} \cos\left(\frac{\pi s}{2}\right) }. 
			\end{align}
			 Next, we write the functional equation for $L(s, \chi)$ \cite[p.~71]{MR1790423}
\begin{align}\label{lq}
L(s, \chi) = \frac{\tau(\chi)}{i^{\kappa}\sqrt{q}}\left(\frac{\pi}{q}\right)^{s-1/2} \frac{\Gamma(\frac{1-s+\kappa}{2})}{\Gamma(\frac{s+\kappa}{2})}  L(1-s, \bar{\chi}),  
\end{align}
where 
			 \begin{align*}
				\kappa=\kappa(\chi)=\begin{cases}
					  0, &\quad \mbox{if} \ \ \chi(-1)=1,\\
					  1, &\quad \mbox{if} \  \ \chi(-1)=-1.\\
				\end{cases}
			\end{align*}
Employing \eqref{relation of gamma} and \eqref{reflection formula for gamma add} in \eqref{lq}, we obtain \cite[Corrolary 10.9, p.~333]{MR2378655}
			\begin{align}\label{ll(s)}
				L(s, \chi) & =i^{-\kappa} \frac{\tau(\chi)}{ \pi}\left(\frac{(2\pi)}{q}\right)^{s} \Gamma(1-s) \sin \frac{\pi (s+\kappa)}{2} L(1-s, \bar{\chi}).  
			\end{align}
			 Now replacing $s$ by $s-z$ in \eqref{ll(s)}, we get
			\begin{align*}
				L(s-z, \chi) & =i^{-\kappa} \frac{\tau(\chi)}{\pi}\left(\frac{(2\pi)}{q}\right)^{s-z} \Gamma(1+z-s) \sin \frac{\pi (s+\kappa-z) }{2} L(1+z-s, \bar{\chi}). 
			\end{align*}
			So, we can rewrite the above equation as
   \begin{align}\label{exact l}
       \Gamma(1+z-s)L(1+z-s, \bar{\chi})= i^{\kappa}\frac{\pi}{\tau(\chi)}\left( \frac{q}{2\pi}\right)^{s-z}\frac{L(s-z, \chi)}{\sin{\pi(\frac{s+\kappa-z}{2})}}.
   \end{align}
	 We will also note that  \cite[p. 69, p. 71]{MR1790423}
			\begin{align} 
				\tau(\chi)\tau( \bar{\chi})= 
				\begin{cases}
					-q, &\quad  \text{for odd primitive }\chi \ mod \ q, \\
					q, &\quad   \text{for even non-principal primitive }\chi \ mod \ q.\label{both}
				\end{cases} 
			\end{align}
   \section{Proof of  results  when \texorpdfstring{$z \in\mathbb{Z}_{+}$}{TEXT}   }\label{proof of integer nu}
		    We will start this section by considering a more general setup. Let $\chi$ be any Dirichlet character modulo $q$ and $z\in \mathbb{C}$. Let $f_{z}(n)$ be one of the arithmetical functions
      $\sigma_{z, \chi }(n)$  or $\Bar{\sigma}_{z,\chi}(n)$  or $\sigma_{z,\chi_1,\chi_2}(n)$ defined in \eqref{defimainpaper}. We denote 
      \begin{align}\label{Dirichletseries}
      F_{z}(s):= \sum_{n=1}^{\infty} \frac{f_{z}(n)}{n^s}, \ \ \Re(s)>1.
      \end{align}
 Hence $F_{z}(s)$ is one of the  Dirichlet series given in \eqref{Lfz1} or \eqref{Lfz2} or \eqref{Lfz3}. As mentioned in the previous section, we will consider $\Re(\nu)> 0$ and $\nu=0$. Employing   Lemma \ref{eq:bessel} with $t=n$ and subsequently interchanging the summation and integration, we get
\begin{align}\label{change_summation_integration_00}
				\sum_{n=1}^{\infty} f_{z}(n) n^{\nu/2} K_{\nu}(a\sqrt{nx}) & = \frac{1}{2} \left(\frac{2}{a \sqrt{x}}\right)^{ \nu}  \frac{1}{2\pi i} \int_{(c)}  \Gamma(s) \Gamma(s + \nu) \left( \frac{4}{a^2 x } \right)^s  \sum_{n=1}^{\infty} f_{z}(n) n^{-s} \mathrm{d}s \nonumber \\ 
				& = \frac{1}{2}  X^{\nu/2} \frac{1}{2\pi i} \int_{(c)}  \Gamma(s) \Gamma(s + \nu)   F_{z}(s) X^s \mathrm{d}s,
			\end{align}
			where $c>\Re{(z)}+1$ and $X = \frac{4}{a^2 x}$. Here the notation $(c)$ denotes the vertical line $[c-i\infty, c+i\infty]$. Next, we investigate the following integral
			\begin{align}\label{rr1}
				I_{z}^{(\nu)}(X):=\frac{1}{2\pi i} \int_{(c)}  \Gamma(s + \nu)  \Gamma(s) F_{z}(s) X^s \mathrm{d}s.
			\end{align}
We shall use the Cauchy residue theorem to evaluate this line integral in \eqref{rr1}. We consider the contour formed by the line
segments $ [c- i T, c+ i T], [c+ i T, -d + i T], [-d + i T, -d - i T], [-d - iT, c-iT]$, where the choice for $d$ is as follows: $ 0< d< \min \{1, \Re(\nu)\}$ whenever $\Re(\nu)>0$ and $0<d<1$ otherwise. Here, $T$ is taken to be a large positive number. The possible poles of the integrand function in \eqref{rr1} are at $s=0, 1$ and $z+1$. Now letting $T \rightarrow \infty$ and invoking Lemma \ref{Gamma}, one can show that the integrals along the horizontal segments $[c+iT, -d+iT]$ and $[-d-iT, c-iT]$ vanish and get\begin{align}\label{applying_CRT_00}
				I_{z}^{(\nu)}(X) &= R_{z+1}+R_1+R_0+\frac{1}{2\pi i} \int_{(-d)}  \Gamma(s + \nu)  \Gamma(s)  F_{z}(s)  X^s \mathrm{d}s, 
			\end{align}
where $R_{z+1}$, $R_1$ and $R_0$ are the residues at $s=z+1,1$ and $s=0$, respectively. It is easy to see that $R_{z+1}=0$ whenever $z=0$.  Hence combining \eqref{change_summation_integration_00} and \eqref{rr1} together with \eqref{applying_CRT_00}, we obtain
\begin{align}\label{K_nuformula}
 \sum_{n=1}^{\infty} f_{z}(n) n^{\nu/2} K_{\nu}(a\sqrt{nx}) = \frac{1}{2} X^{\nu/2} \left(R_{z+1}+R_1+R_0+J_{z}^{(\nu)}(X)\right),  
\end{align}
where $J_{z}^{(\nu)}(X)$ is defined by
\begin{align}\label{defiJ}
J_{z}^{(\nu)}(X):= \frac{1}{2\pi i} \int_{(-d)}  \Gamma(s + \nu)  \Gamma(s)  F_{z}(s)  X^s \mathrm{d}s.    
\end{align}
Next, we will offer the proofs of the theorems corresponding to $z=k$, where $k$ is a non-negative integer.
	 \vspace{.2cm}	
		 \begin{proof}[Theorem \rm{\ref{thm1}}][]  Letting $f_{k}(n)=\sigma_{k, \chi}(n)$ where $\chi$ being an odd primitive character modulo $q$ and $k$  an even, non-negative integer in \eqref{K_nuformula}, we obtain 
   \begin{align}\label{K_nuformula 11}
    \sum_{n=1}^{\infty} \sigma_{k, \chi}(n) n^{\nu/2} K_{\nu}(a\sqrt{nx}) = \frac{1}{2} X^{\nu/2} \left(R_{k+1}+R_1+R_0+J_{z}^{(\nu)}(X)\right),     \end{align}
    where $\Re(\nu)>0$ and $ J_{k}^{(\nu)}(X)$ is defined in \eqref{defiJ} with $F_{k}(s)=\zeta(s) L(s-k, \chi)$. It is easy to see that $R_{k+1}=0$ as the integrand function in \eqref{defiJ} does not have any pole at $s=k+1$. Here, one can notice that $L(s-k, \chi)$ has a zero at $s=1$ when $k\geq 2$ is an even integer and $\chi$ is odd. Therefore, we will not get any contribution from the pole of $\zeta(s)$ at $s=1$. However, if $k=0$, the integrand in \eqref{defiJ} will encounter a pole at $s=1$. Therefore, we can get
   \begin{align}
     R_1=  \begin{cases}
      0,&\ \text{if} \ k>0,\\
          \Gamma(1+\nu)L(1,\chi) X,  & \ \text{if} \ k=0.
       \end{cases}\label{R_1FOR0}
   \end{align}
  The integrand also has a pole at $s=0$ with residue $R_0$ given by 
   \begin{align} 
				R_0 =- \frac{\Gamma(\nu) L(-k, \chi)}{2}  
				=\frac{(-1)^{\frac{k}{2}}i \tau(\chi)}{2{ \pi}}\left(\frac{(2\pi)}{q}\right)^{-k} \Gamma(1+k) \Gamma(\nu)L(1+k, \bar{\chi}), \label{r_0}
			\end{align}
where in the last step, we have applied functional equation \eqref{ll(s)}. Collecting \eqref{R_1FOR0} and \eqref{r_0} and $R_{k+1}=0$ and then substituting them in \eqref{K_nuformula 11}, we get
    \begin{align}\label{K_nuformula1}
X^{-\frac{\nu}{2}}	\sum_{n=1}^\infty\sigma_{k,\chi}(n)n^{\frac{\nu}{2}}K_\nu(a\sqrt{nx})=& \frac{(-1)^{\frac{k}{2}}i \tau(\chi)}{4{ \pi}}\left(\frac{(2\pi)}{q}\right)^{-k} \Gamma(1+k) \Gamma(\nu) L(1+k,\Bar{\chi})\nonumber \\
&+\delta_k \frac{\Gamma(1+\nu)L(1,\chi)}{2} X +\frac12 J_{k}^{(\nu)}(X),        \end{align}
where $\delta_k$ is defined in \eqref{d_k number}. To evaluate $J_{k}^{(\nu)}(X)$ defined in \eqref{defiJ}, we invoke the functional equations  \eqref{1st_use} and \eqref{ll(s)},
	\begin{align*}
				J_{k}^{(\nu)}(X) &=\frac{h_k}{2\pi i} \int_{(-d)} \Gamma(s + \nu)  \Gamma(1+k-s)  \zeta(1-s) L(1-s+k, \bar{\chi} ) Y^s \mathrm{d}s \nonumber \\
				&=\frac{ Y h_k}{2\pi i} \int_{(1+d)}  {\Gamma(1-s + \nu)  \Gamma(k+s)  \zeta(s) L(s+k, \bar{\chi} )}  Y^{-s} \mathrm{d}s \nonumber \\
				&= Y h_k\sum_{n=1}^{\infty} \sigma_{-k, \bar{\chi}}(n) \   \frac{1}{2\pi i} \int_{(1+d)} {\Gamma(1-s + \nu)  \Gamma(k+s)}  (nY)^{-s} \mathrm{d}s, 
			\end{align*}
   where $h_k=\frac{ (-1)^{1+\frac{k}{2}}i\tau(\chi)}{2 \pi} \left(\frac{q}{2\pi}\right)^{k}$ and $Y=\frac{4\pi^2}{q}X$ with $X=\frac{4}{a^2 x}$. As $0<d<\Re(\nu)$, we can apply Lemma \ref{lem4} with $n =k$ and $a=1+\nu$ to obtain
			\begin{align}\label{k odd}
				J_{k}^{(\nu)}(X) &=Y^{k+1} \Gamma(1+\nu) h_k\sum_{n=1}^{\infty} \sigma_{-k, \bar{\chi}}(n) \frac{(\nu+1) \cdots (\nu+k) n^{k} }{(1+nY)^{1+\nu+k} }\nonumber \\
				&=Y^{k+1} \Gamma(1+\nu+k) h_k\sum_{n=1}^{\infty} \frac{ \Bar{\sigma}_{k,\chi}(n) }{(1+nY)^{1+\nu+k}},
			\end{align}
		where in the penultimate step we have used the fact $  ~~  n^k\sigma_{-k,\chi}(n)=\Bar{\sigma}_{k,\chi}(n)$. Therefore, remarking $Y=\frac{16\pi^2}{a^2 q x}$ and inserting \eqref{k odd} in \eqref{K_nuformula1} and simplifying, we can complete the proof. 
			\end{proof}	
 \begin{proof}[Theorem \rm{\ref{odd_k0thm1}}][]
Let us begin the proof by taking $f_{k} (n)=\sigma_{k, \chi}(n)$ with $\chi$ being an odd primitive character modulo $q$ and $k\geq 0$ an even integer and $\nu=0$ in \eqref{K_nuformula}. The corresponding Dirichlet series, in this case, is $F_{k}(s)=\zeta(s) L(s-k, \chi)$. Therefore, we obtain
\begin{align}\label{K_nuformula12}
\sum_{n=1}^\infty\sigma_{k,\chi}(n) K_0(a\sqrt{nx})=\frac12(R_{k+1}+R_1+R_0+J_{k}^{(0)}(X)),
\end{align}
where $ J_{k}^{(0)}(X)$ is defined in \eqref{defiJ}.
 It is clear that $R_{k+1}=0$ for $k\geq0$. $L(s-k, \chi)$ has a zero at $s=1$ in case $k\geq 2$ is an even integer, and $\chi$ is odd. So we will not get any contribution from the pole of $\zeta(s)$ at $s=1$. However, in the case of $k=0$, the integrand in \eqref{defiJ} will encounter a pole at $s=1$. Hence, we can write 
\begin{align}\label{R111_0}
R_1=\begin{cases}
0,  & \ \    \text{if }  k>0,\\
 L(1,\chi)X,  & \ \    \text{if }  k=0,
\end{cases}
\end{align}
 and the integrand in \eqref{defiJ}  encounters a double pole at $s=0$ with residue $R_0$ given by
\begin{align}\label{00r0}
    R_0=-\frac{ L(-k,\chi) }{2}\left(  \log(2\pi X) +\frac{ L^\prime(-k,\chi)}{ L(-k,\chi)} -2\gamma \right).
\end{align}
Now using \eqref{R111_0} and \eqref{00r0} and the fact $R_{k+1}=0$ in \eqref{K_nuformula12}, we obtain
\begin{align}\label{K_nuformula211}
  \sum_{n=1}^\infty\sigma_{k,\chi}(n) K_0(a\sqrt{nx})=&\frac{\delta_k}{2}  L(1,\chi)X -\frac{ L(-k,\chi) }{4}\left(  \log(2\pi X) +\frac{ L^\prime(-k,\chi)}{ L(-k,\chi)} -2\gamma \right)+\frac12 J_{k}^{(0)}(X), 
\end{align}
where $\delta_k$ is defined in \eqref{d_k number}.
For $J_{k}^{(0)}(X)$, we employ the functional equations \eqref{1st_use} and \eqref{ll(s)} to obtain 
	\begin{align}\label{J_k11}
				J_{k}^{(0)}(X) &=\frac{h_k}{2\pi i} \int_{(-d)} \Gamma(s )  \Gamma(1+k-s)  \zeta(1-s) L(1-s+k, \bar{\chi} ) Y^s \mathrm{d}s \nonumber \\
				&=\frac{ Y h_k}{2\pi i} \int_{(1+d)}  {\Gamma(1-s  )  \Gamma(k+s)  \zeta(s) L(s+k, \bar{\chi} )}  Y^{-s} \mathrm{d}s \nonumber \\
				&=  Y h_k\sum_{n=1}^{\infty} \sigma_{-k, \bar{\chi}}(n) \   \frac{1}{2\pi i} \int_{(1+d)} {\Gamma(1-s )  \Gamma(k+s)}  (n Y)^{-s} \mathrm{d}s  \nonumber \\
      &= \pi  Y h_k\sum_{n=1}^{\infty} \sigma_{-k, \bar{\chi}}(n) \ \frac{ 1}{2\pi i} \int_{(1+d)}  \frac{\Gamma(k+s)}{\Gamma(s) \sin(\pi s)}  (n Y)^{-s} \mathrm{d}s \nonumber \\
    &=\pi  Y h_k \left(\sum_{n\leq Y^{-1}}+\sum_{n>Y^{-1}}\right) \sigma_{-k, \bar{\chi}}(n) \frac{1}{2\pi i} \int_{(1+d)}  \frac{\Gamma(k+s)}{\Gamma(s) \sin(\pi s)} (n Y)^{-s} \mathrm{d}s,
  	\end{align}
   where $h_k=\frac{ (-1)^{1+\frac{k}{2}}i\tau(\chi)}{2 \pi} \left(\frac{q}{2\pi}\right)^{k}$ and $Y=\frac{4\pi^2}{q}X$ with $X=\frac{4}{a^2 x}$. In the second last step, we have used the reflection formula \eqref{reflection formula for gamma add}. 
   \par We will first investigate the infinite sum $\sum_{n>Y^{-1}}$. To evaluate this inner line integral in \eqref{J_k11}, we shall use the Cauchy residue theorem with the contour consisting of the line segments  $ [1+d- i T, 1+d+ i T], [1+d+ i T, M+\frac{1}{2}+ i T], [M+\frac{1}{2} + i T,  M+\frac{1}{2} - i T], [M+\frac{1}{2} - iT, 1+d-iT]$ where $M \in \mathbb{N}$ is a large number and $T$ is a large positive number. The poles of the integrand function in \eqref{J_k11} are at $2$, $3, \cdots$, $M$, and they are simple. The residue at $s=m$ is given by 
  \begin{align}\label{residues_2}
\mathcal{R}_m:= \frac{1}{\pi}(-1)^m m (m+1)...  (m+k-1) (n Y)^{-m},    
  \end{align}
 where $m=2$, $3, \cdots$, $M$. Employing Lemma \ref{Gamma}, we can show that both the integrals along the horizontal lines $[1+d+iT, M+\frac12+iT]$ and $[M+\frac12-iT, 1+d-iT]$  vanish as $T \rightarrow \infty$. From \eqref{residues_2}, we arrive at
 \begin{align*}
      \frac{1}{2\pi i} \int_{(1+d)}  \frac{\Gamma(k+s)}{\Gamma(s) \sin(\pi s)}  (n Y)^{-s} \mathrm{d}s =&-\sum_{m=2}^M \mathcal{R}_m+  \frac{1}{2\pi i} \int_{(M+\frac12)}  \frac{\Gamma(k+s)}{\Gamma(s) \sin(\pi s)} (n Y)^{-s}\mathrm{d}s \nonumber\\
      = &-\frac{1}{\pi }\sum_{m=2}^M  (-1)^m m(m+1)...  (m+k-1) (n Y)^{-m}+O_k\left(  \frac{M^{k}}{\left( nY\right)^{M+1/2}}  \right),\notag
    \end{align*}
where we have used $|\sin \pi (\sigma+it)|\gg e^{\pi |t|}$ for $|t|\geq 1$ to bound the integral $\int_{(M+\frac12)}$ and the implied constant depends on $k$. Next, allowing $M \rightarrow \infty$, the error term goes to $0$ as $n>Y^{-1}$. Now simplifying, we readily obtain that 
 \begin{align*}
      \frac{1}{2\pi i} \int_{(1+d)}  \frac{\Gamma(k+s)}{\Gamma(s) \sin(\pi s)} (n Y)^{-s}\mathrm{d}s =&-\frac{k!}{\pi } \left(  \frac{1}{nY} \right) \left(1-\frac{n^{k+1}}{(Y^{-1}+n)^{k+1}}\right),  
      \end{align*}
and we easily deduce from the above expression that
      \begin{align}\label{sigma1}
     \sum_{n>Y^{-1}} \sigma_{-k, \bar{\chi}}(n) \frac{1}{2\pi i} \int_{(1+d)}  \frac{\Gamma(k+s)}{\Gamma(s) \sin(\pi s)} (n Y)^{-s} \mathrm{d}s =&-\frac{k!}{\pi Y } \sum_{n>Y^{-1}} \frac{\sigma_{-k, \bar{\chi}}(n)}{n} \left(1-\frac{n^{k+1}}{(Y^{-1}+n)^{k+1}}\right).  
      \end{align}
Similarly, by shifting the line of integration to the left, we obtain
			\begin{align}\label{sigma2}
	 \sum_{n\leq Y^{-1}} \sigma_{-k, \bar{\chi}}(n) \frac{1}{2\pi i} \int_{(1+d)}  \frac{\Gamma(k+s)}{\Gamma(s) \sin(\pi s)} (n Y)^{-s} \mathrm{d}s &=-\frac{k!}{\pi Y } \sum_{n\leq Y^{-1}} \frac{\sigma_{-k, \bar{\chi}}(n) }{n} \left(1-\frac{n^{k+1}}{(Y^{-1}+n)^{k+1}}\right).  
			\end{align}
Inserting \eqref{sigma1} and \eqref{sigma2} in \eqref{J_k11},
			\begin{align}\label{rami} 
				J_{k}^{(0)}(X)=-k! h_k \sum_{n=1}^\infty \frac{\sigma_{-k, \bar{\chi}}(n)}{n}  \left(1-\frac{n^{k+1}}{(Y^{-1}+n)^{k+1}}\right).  
			\end{align}
We finish the proof by noting $Y=\frac{16\pi^2}{a^2 q x}$, substituting \eqref{rami} in \eqref{K_nuformula211} and then simplifying.
			 \end{proof}
 \begin{proof}[Theorem \rm{\ref{thm3}}][]
   Here we will take  $f_{k}(n)=\Bar{\sigma}_{k, \chi}(n)$ and $\Re{(\nu)}>0$ in \eqref{K_nuformula}. Similar to the previous theorems, $\chi$ is  odd and $k\geq 2$ is an even integer, and the corresponding Dirichlet series is $F_{k}(s)=\zeta(s-k) L(s, \chi)$.  It is clear that $R_{1}=0$. When $k\geq 2$ is an even integer, $\zeta(s-k)$ has a zero at $s=0$. Therefore, we will not get any contribution from the pole of $\Gamma(s)$ at $s=0$. But the integrand in \eqref{defiJ} will encounter a pole at $s=k+1$ with the residue $R_{k+1}$ given by
  \begin{align*}
R_{k+1}&=\Gamma(k+1)\Gamma(\nu+k+1)L(k+1,\chi)X^{k+1}.
    \end{align*}
  The calculation for $J_{k}^{(0)}(X)$ will be similar as given in the proof of Theorem \rm{\ref{thm1}}. To avoid repetition, we skip the detail of the proof.
  \end{proof}
\begin{proof}[Theorem \rm{\ref{1thm1}}][]
 Here we will consider $f_{k}(n)=\Bar{\sigma}_{k,\chi}(n)$ with $\chi$ being an odd primitive character modulo $q$ and $k\geq 2$ an even integer and $\nu=0$ in \eqref{K_nuformula}. We skip the detail of the proof because of its similarity with the proof of Theorem \ref{odd_k0thm1}.
    \end{proof}
 %
   \vspace{.1cm}
 \noindent\textit{Proofs of Theorems} \rm{\ref{thmeven1}} and \rm{\ref{7}}. Here, we will take $f_{k}(n)= {\sigma}_{k,\chi}(n)$ and $\chi$ being a non-principal even primitive Dirichlet character modulo $q$ and $k\geq 1$ an odd integer in \eqref{K_nuformula}. We can see that Theorem \ref{thmeven1} deals with the case $\Re(\nu)>0$ while Theorem \ref{7} concerns with $\nu=0$. Proceeding by almost identically the same argument as in the proof of Theorems \ref{thm1} and \rm{\ref{odd_k0thm1}}, one can deduce Theorems \ref{thmeven1} and \rm{\ref{7}}, respectively. We leave the details of the proofs for the reader.
 \hfill \qed
\vspace{0.1cm}\\
\noindent\textit{Proofs of Theorems} \rm{\ref{thmeven2}} and \rm{\ref{10}}. The proofs  are similar to the corresponding proofs of Theorems \ref{thm3} and \ref{1thm1} for the odd character.\hfill\qed
\vspace{0.1cm}
\begin{proof}[Theorem \rm{\ref{sigmanu=0}}][]
 It deals with the special case $k=0$ and $\nu=0$ when $\chi$ is a non-principal even primitive character modulo $q$. Thus setting $f_{0}(n)=d_{\chi}(n)$ in \eqref{K_nuformula}, we obtain
\begin{align} \label{w4}
	 \sum_{n=1}^\infty d_{\chi}(n) K_0(a\sqrt{nx})=\frac12(R_1+R_0+J_{0}^{(0)}(X)),
	\end{align}
where the residues $R_1$ and $R_0$ are given by
\begin{align}
R_1=& L (1,\chi)X,\label{w3}\\
R_0=&-\frac{1}{2}L^\prime(0,\chi)=-\frac{\tau(\chi)}{4 }L (1,\bar{\chi}) \label{w2},
\end{align}
where  in \eqref{w2}, we have used \cite[p.~181,  equation (3.2)]{MR749681}. Next, we evaluate $J_{0}^{(0)}(X)$ defined in \eqref{defiJ} with $F_{0}(s)=\zeta(s) L(s, \chi)$. Utilizing the functional equations \eqref{1st_use} and \eqref{ll(s)}, one can get
\begin{align}\label{w11} 
J_{0}^{(0)}(X)=\frac{\tau(\chi) }{4}Y \sum_{n=1}^\infty d_{\Bar{\chi}}(n)  \frac{1}{2\pi i }\int_{(1+d)}\frac{(nY)^{-s}}{\sin^2(\pi s/2)}\mathrm{d}s ,
\end{align}
 where $Y= \frac{4\pi^2 X}{q}$. As $0<d<1$, applying inverse Mellin transform to \eqref{sinpi2} of Lemma \ref{sinsquare} and then employing the formula in \eqref{w11}, we deduce that
 \begin{align}\label{w1}
   J_{0}^{(0)}(X)= \frac{\tau(\chi) }{\pi^2 }Y \sum_{n=1}^\infty d_{\Bar{\chi}}(n)  \frac{\log (nY)}{(nY)^2-1}.
 \end{align}
 Inserting \eqref{w3}, \eqref{w2} and \eqref{w1} in \eqref{w4} and noting $Y=\frac{16\pi^2}{a^2 q x}$, one can complete the proof.
 \end{proof}
 Next, we are going to investigate the identities involving two characters. 
   \begin{proof}[Theorem \rm{\ref{evenoddthm1}}][]
We will take $f_{k}(n)=\sigma_{k,\chi_1,\chi_2}(n)$ and $k\geq 1$  an odd integer and $\Re(\nu)>0$ in \eqref{K_nuformula}. By assumption, $\chi_1$ and $\chi_2$ are primitive characters modulo $ p$ and $q$, respectively, such that either both are non-principal even characters or both are odd characters. In the notation of \eqref{Dirichletseries}, $ F_{k}(s)=L(s-k,\chi_1) L(s, \chi_2)$.  We get \begin{align} \label{s100}
	 \sum_{n=1}^\infty\sigma_{k,\chi_1,\chi_2}(n) K_{\nu}(a\sqrt{nx})=\frac{1}{2}X^{\nu/2}(R_{k+1}+R_1+R_0+J_{k}^{(\nu)}(X)),
			\end{align} 
where $J_{k}^{(\nu)}(X)$ is defined in \eqref{defiJ}. It is easy to see that $R_{k+1}=0$ and $R_1=0$. When both  $\chi_1$ and $\chi_2$ are non-principal even primitive characters, $L(s, \chi_2)$ has a zero at $s=0$. Hence we will not be getting any contribution from the pole of $\Gamma(s)$ at $s=0$. As a result, we will get $R_0=0$. If both $\chi_1$ and $\chi_2$ are odd primitive characters, $L(s-k, \chi_1)$ has a zero at $s=0$ since $k$ is an odd integer. Again, there will be no contribution of the pole from $\Gamma(s)$ at $s=0$. Therefore $R_0=0$. Now utilizing the facts $R_{k+1}=0$, $R_1=0$ and $R_0=0$ in \eqref{s100}, we obtain
\begin{align} \label{s1}
	 \sum_{n=1}^\infty\sigma_{k,\chi_1,\chi_2}(n) K_{\nu}(a\sqrt{nx})=\frac12X^{\nu/2}J_{k}^{(\nu)}(X).
			\end{align} 
To evaluate $J_{k}^{(\nu)}(X) $, we utilize the functional equations \eqref{ll(s)}, \eqref{exact l} with \eqref{both}
	 \begin{align}\label{s2}
				J_{k}^{(\nu)}(X) 
				&= Y g_k\sum_{n=1}^{\infty} \sigma_{-k, \bar{\chi_1},\bar{\chi_2}}(n) \   \frac{1}{2\pi i} \int_{(1+d)}  \Gamma(1-s + \nu)  \Gamma(k+s)  (nY)^{-s} \mathrm{d}s ,
			\end{align}
    where $g_k=\frac{ (-1)^{ \frac{k+1}{2}}p^k\tau(\chi_1) \tau(\chi_2)}{(2 \pi)^{k+1}}  $ and $Y=\frac{4\pi^2}{pq}X$. As $0<d<1$, appealing to Lemma \ref{lem4} with $n =k$ and $a=1+\nu$, we deduce  
    \begin{align}\label{s3}
     J_{k}^{(\nu)}(X)= Y^{k+1} g_k \Gamma(1+\nu+k)\sum_{n=1}^{\infty}   \frac{ \sigma_{k,  \bar{\chi_2},\bar{\chi_1} }(n)}{{(1+nY)}^{1+\nu+k} },  
    \end{align}
   where we have used the fact  $\sigma_{-k, \bar{\chi_1}, \bar{\chi_2}}(n)=n^{-k}\sigma_{k,  \bar{\chi_2},\bar{\chi_1} }(n) $. We complete the proof by  substituting \eqref{s3} in \eqref{s1} and remarking $Y=\frac{16\pi^2}{a^2 pq x}$.
\end{proof}
\begin{proof}[Theorem \rm{\ref{evenoddthm1nu=0}}][]
We leave the proof to the reader for its similarity with the proofs of Theorems \ref{odd_k0thm1} and \ref{7}. 
 \end{proof}
\noindent\textit{Proofs of Theorems} \rm{\ref{EVENEVENNU=0k=0}} and \rm{\ref{ODDODDnu0k0}}.    
We begin the proof by setting $k=0$ and $\nu=0$ and $f_{0}(n)=d_{\chi_1, \chi_2}(n)$ in \eqref{K_nuformula}. This will give 
\begin{align}\label{s11}
   \sum_{n=1}^\infty d_{\chi_1,\chi_2}(n)  K_0(a\sqrt{nx})= & \frac12(R_0+R_1+J_{0}^{(0)}(X)),
\end{align}
 where $J_{0}^{(0)}(X))$ is defined in \eqref{defiJ} with $F_0(z)=L(s, \chi_1)L(s, \chi_2)$. Here we will have $R_1=0$. Now we will discuss the following two cases.\\
 {\it Case 1: When $\chi_1 $ and $\chi_2$ are even non-principal primitive characters modulo $ p$ and $q$, respectively.} 
 Both $L(s,\chi_1)$ and $L(s,\chi_2)$ have simple zero at $s=0$ which will get cancelled by the double pole of $\Gamma^2(s)$ at $s=0$. Hence we have $R_0=0$. Employing functional relation \eqref{ll(s)}, we obtain
\begin{align}\label{doneearlier}
    J_{0}^{(0)}(X)= \frac{\tau(\chi_1)\tau(\chi_2) }{4}Y \sum_{n=1}^\infty d_{\Bar{\chi}_1,\Bar{\chi}_2}(n) \frac{1}{2\pi i } \int_{(1+d)}  \frac{(n Y)^{-s}}{\sin^2\left(\frac{\pi s}{2}\right)}\mathrm{d}s ,
\end{align}
 where $Y=\frac{4\pi^2}{pq}X$. Note that integral in \eqref{doneearlier} can be treated similarly as in the proof of Theorem \rm{\ref{sigmanu=0}}. To avoid repetitions, we skip the detail. \\
{\it Case 2: When $\chi_1 $ and $\chi_2$ are odd primitive characters modulo $ p$ and $q$, respectively.}
  In this case,  the integrand will encounter a double pole at $s=0$. Hence the residue $R_0$ is given by
  \begin{align}\label{ss2}
				 R_0= L (0,\chi_1)L (0,\chi_2) \left( -2\gamma+  \log  \left( X \right)    +\frac{ L ^\prime(0,\chi_1)}{ L (0,\chi_1)}+\frac{L^\prime (0,\chi_2) }{L (0,\chi_2) } \right). 
    \end{align}
   Employing \eqref{ss2} and $R_1=0$ in \eqref{s11}, we obtain 
   \begin{align}\label{s110}
   \sum_{n=1}^\infty d_{\chi_1,\chi_2}(n)  K_0(a\sqrt{nx})= & \frac12L (0,\chi_1)L (0,\chi_2) \left(-2\gamma+  \log  \left( X \right)    +\frac{ L ^\prime(0,\chi_1)}{ L (0,\chi_1)}+\frac{L^\prime (0,\chi_2) }{L (0,\chi_2) }\right)+\frac{1}{2}J_{0}^{(0)}(X) .
\end{align}
 Now appealing to functional relation \eqref{ll(s)}, we will have
 \begin{align} \label{integralev}
     J_{0}^{(0)}(X)= &-\frac{\tau(\chi_1)\tau(\chi_2) }{4 }Y \sum_{n=1}^\infty d_{\Bar{\chi}_1,\Bar{\chi}_2}(n)  \frac{1}{2\pi i } \int_{(1+d)}  \frac{(nY)^{-s}}{\cos^2\left(\frac{\pi s}{2}\right)}\mathrm{d}s  \nonumber  \\
     =&-\frac{\tau(\chi_1)\tau(\chi_2) Y}{4 } \left( \sum_{n <Y^{-1}}+\sum_{n>Y^{-1}}\right) d_{\Bar{\chi}_1,\Bar{\chi}_2}(n) \frac{1}{2\pi i } \int_{(1+d)}  \frac{\left(  {nY}\right)^{-s}}{\cos^2\left(\frac{\pi s}{2}\right)}   \mathrm{d}s ,
\end{align}   where $Y=\frac{16\pi^2}{a^2 pq x}$ and $Y^{-1} \notin \mathbb{Z}_+$. \\
We first evaluate the inner line integral on the sum $\sum_{n>Y^{-1}} $. We shall use the Cauchy residue theorem with the contour formed by the lines $ [1+d- i T, 1+d+ i T], [1+d+ i T, M+\frac{1}{2}+ i T], [M+\frac{1}{2} + i T, M+\frac{1}{2} - i T], [M+\frac{1}{2} - iT, 1+d-iT]$ where $M \in \mathbb{N}$ is any odd large number and $T$ is a large positive number. The poles of the integrand function in \eqref{integralev} are at $ 3,5, \cdots, M$, and they are double poles. The residue at $s=m$ is given by 
  \begin{align}\label{residues_22}
\mathcal{R}_m:= -\frac{4}{\pi^2}\left({nY}\right)^{-m} \log \left( {nY}\right),    
  \end{align}
 where $m=3$, $5, \cdots$, $M$. Employing Lemma \ref{Gamma}, we can show both the integrals along the horizontal line segments $[1+d+iT, M+\frac12+iT]$ and $[M+\frac12-iT, 1+d-iT]$ vanish as $T \rightarrow \infty$. Utilising \eqref{residues_22}, we arrive at
\begin{align}
  \frac{1}{2\pi i } \int_{(1+d)}  \frac{(n Y)^{-s}}{\cos^2\left(\frac{\pi s}{2}\right)}   \mathrm{d}s &=-\sum_{m=1}^{  \frac{M-1}{2} } \mathcal{R}_{2m+1}+  \frac{1}{2\pi i } \int_{(M+\frac12)}  \frac{(n Y)^{-s}}{\cos^2\left(\frac{\pi s}{2}\right)}  \mathrm{d}s \notag\\
     &=\frac{4}{\pi^2}\sum_{m=1}^{  \frac{M-1}{2} } \frac{ \log \left( {nY}\right)}{ \left( {nY}\right)^{2m+1}}+O\left(    \frac{1}{\left(nY\right)^{M+1/2}}\right).\label{iddd} \end{align}
 Letting $M \rightarrow \infty$, the error term in \eqref{iddd} goes to $0$ as $n>Y^{-1}$. Thus simplifying, we can readily deduce  that 
\begin{align*}
    \frac{1}{2\pi i } \int_{(1+d)}  \frac{(n Y)^{-s}}{\cos^2\left(\frac{\pi s}{2}\right)}  \mathrm{d}s  =&\frac{4}{\pi^2}\sum_{m=1}^\infty \frac{ \log \left( {nY}\right)}{ \left( {nY}\right)^{2m+1}}=\frac{4}{\pi^2 Y^3} \frac{1}{n(n^2-Y^{-2})}\log \left({nY}\right),
\end{align*}
 and subsequently, we get
\begin{align}\label{k11}
    \sum_{n>Y^{-1}}  d_{\Bar{\chi}_1,\Bar{\chi}_2}(n) \frac{1}{2\pi i } \int_{(1+d)}  \frac{(n Y)^{-s}}{\cos^2\left(\frac{\pi s}{2}\right)} \mathrm{d}s   =\frac{4}{\pi^2}\sum_{n>Y^{-1}}  d_{\Bar{\chi}_1,\Bar{\chi}_2}(n) \frac{Y^{-3}}{n(n^2-Y^{-2})}\log \left( {nY}\right).
\end{align}
 Similarly, by shifting the integration line to the left, we get
  \begin{align}\label{k22}
   \sum_{n< Y^{-1}}  d_{\Bar{\chi}_1,\Bar{\chi}_2}(n) \frac{1}{2\pi i } \int_{(1+d)}  \frac{(n Y)^{-s}}{\cos^2\left(\frac{\pi s}{2}\right)} \mathrm{d}s  =\frac{4}{\pi^2}\sum_{n\leq Y^{-1}}  d_{\Bar{\chi}_1,\Bar{\chi}_2}(n)  \frac{Y^{-3}}{n(n^2-Y^{-2})}\log \left( {nY}\right).  
  \end{align}
   Hence combining \eqref{k11} and \eqref{k22} with \eqref{integralev}, we obtain
  \begin{align}\label{ss3}
     J_{0}^{(0)}(X)= -\frac{\tau(\chi_1)\tau(\chi_2) }{\pi^2 } \sum_{n=1}^\infty d_{\Bar{\chi}_1,\Bar{\chi}_2}(n)   \frac{Y^{-2}}{n(n^2-Y^{-2})}\log \left( {nY}\right).
\end{align}
 Inserting \eqref{ss3} in \eqref{s110} and remarking that $Y=\frac{16\pi^2}{a^2 pq x}$, we get the desired result.
 \hfill\qed
  \vspace{0.1cm}\\
 \noindent\textit{Proofs of Theorems}  \rm{\ref{evenoddthm2}} and  \rm{\ref{18}}.  The proofs are similar to the proofs of Theorems \rm{\ref{evenoddthm1}} and  \rm{\ref{evenoddthm1nu=0}}. 
  \hfill\qed
  \section{ Proof  of Cohen-Type Identities}\label{proof of cohen identities...}
  This section is devoted to the proof of Cohen-type identities.
Throughout this section, we will deal with $z=-\nu \notin \mathbb{Z}$ and $\Re(\nu)\geq 0$. 
In general set up, if we consider $\Re(\nu)\geq 0$ with $\nu \notin \mathbb{Z}$ and $a=4\pi$, then \eqref{K_nuformula} becomes
 \begin{align}\label{K_nucohen}
 \sum_{n=1}^{\infty} f_{-\nu}(n) n^{\nu/2} K_{\nu}(4\pi\sqrt{nx}) = \frac{1}{2} X^{\nu/2} \left(R_{1-\nu}+R_1+R_0+R_{-\nu}+J_{-\nu}^{(\nu)}(X)\right),  
\end{align}
where $R_{-\nu}$ is the residue corresponding to the pole $s=-\nu$. It is easy to see that the pole at $s=-\nu$ appears if $\Re(\nu)=0$ and $\nu \notin \mathbb{Z}$. The expression $J_{-\nu}^{(\nu)}(X)$ defined in \eqref{defiJ}, can be rewritten as
\begin{align}\label{defiJcohen}
J_{-\nu}^{(\nu)}(X):= \frac{1}{2\pi i} \int_{(-d)}  \Gamma(s + \nu)  \Gamma(s)  F_{-\nu}(s)  X^s \mathrm{d}s,    \end{align} 
where $F_{-\nu}(s)$ defined in \eqref{Dirichletseries} is the Dirichlet series associated with the arithmetical function $f_{-\nu}(n)$. We will note that $X=\frac{1}{4\pi^2 x}$.
		\begin{proof}[Theorem \rm{\ref{evencohen}}][] 	Letting  $f_{-\nu}(n)=\sigma_{-\nu, \bar{\chi}}(n)$  where $\chi$ being a non-principal even primitive character modulo $q$ in \eqref{K_nucohen}, we obtain
			    \begin{align}\label{twistedLLP}
  \sum_{n=1}^{\infty} \sigma_{-\nu, \bar{\chi}}(n) n^{\nu/2}K_{\nu}(4\pi\sqrt{nx})  &=\frac12X^{\nu/2} ( R_{1-\nu}+ R_{1}+R_{0}+R_{-\nu}+J_{ -\nu}^{(\nu)}(X)),\end{align}
			where $J_{ -\nu}^{(\nu)}(X)$ is given in \eqref{defiJcohen} and $F_{-\nu}(s)=\zeta(s) L(s+\nu, \bar{\chi})$. The integrand in \eqref{defiJcohen} will encounter simple poles at $s=1$ and $s=0$ with residues $R_1$ and $R_0$ given by
			\begin{align}\label{residueLL}
   R_{1} &= \Gamma(1+\nu) L(1+\nu, \bar{\chi}) X, \ \ \mbox{and} \ \
				R_0 =-\frac{ \Gamma(\nu) L(\nu, \bar{\chi})}{2},
			\end{align}
			respectively. It is easy to see that $R_{1-\nu}=0$. As $\bar{\chi}$ is a non-principal even primitive character, $L(s+\nu, \bar{\chi})$ has a zero at $s=-\nu$. Therefore, we will not be getting any contribution from the pole of $\Gamma(s+\nu)$ at $s=-\nu$. Hence $R_{-\nu}=0$. Now employing \eqref{residueLL} together with the facts $R_{1-\nu}=0$ and $R_{-\nu}=0$ in \eqref{twistedLLP}, we obtain
    \begin{align}\label{twistedLL}
  \sum_{n=1}^{\infty} \sigma_{-\nu, \bar{\chi}}(n) n^{\nu/2}K_{\nu}(4\pi\sqrt{nx})  &=\frac12X^{\nu/2} \left(\Gamma(1+\nu) L(1+\nu, \bar{\chi}) X-\frac{ \Gamma(\nu) L(\nu, \bar{\chi})}{2}+J_{ -\nu}^{(\nu)}(X)\right),\end{align}
where $X=\frac{1}{4\pi^2 x}$. Next, we evaluate the following integral $J_{ -\nu}^{(\nu)}(X)$. Replacing $s$ by $1-s$ and then employing \eqref{1st_use}, \eqref{exact l}, we obtain \begin{align}\label{jnuprime1}
				J_{-\nu}^{(\nu)}(X) &=\frac{1}{2\pi i} \int_{(1+d)}  \Gamma(1-s+\nu) \Gamma(1-s)   \zeta(1-s) L(1-s+\nu, \bar{\chi}) X^{1-s} \mathrm{d}s \nonumber \\
				&= \left(\frac{2 \pi}{q}\right)^{\nu} \frac{\pi^{2} X }{\tau(\chi)}
   \frac{1}{2\pi i} \int_{(1+d)}  
   \frac{\zeta(s)L(s-\nu,  \chi) } {((2 \pi)^{2} 
   q^{-1}X)^s \sin\left(\frac{\pi s}{2}\right)  \sin\left(\frac{\pi (s-\nu)}{2}\right)}  \mathrm{d}s \nonumber \\
				&=  \frac{(2 \pi)^\nu q^{-\nu}}{4x\tau(\chi)} \frac{1}{2\pi i} \int_{(1+d)}   \frac{\zeta(s)L(s-\nu, \chi) (qx)^{s}}{ \sin\left(\frac{\pi s}{2}\right)  \sin\left(\frac{\pi (s-\nu)}{2}\right)}  \mathrm{d}s.  
			\end{align} 
To evaluate the integral in \eqref{jnuprime1},
 we employ the Cauchy residue theorem with the rectangular contour formed by the lines $ [1+d- i T, 1+d+ i T], [1+d+ i T, 2N+\delta + i T], [2N+\delta + i T, 2N+\delta - i T], [2N+\delta - iT, 1+d-iT]$ with $ N\geq  [\frac{\Re(\nu)+1}{2}] $ and $\{\Re(\nu)+1\}<\delta<1$ and $T$ is a large positive number. One can note that  the simple poles of $\sin^{-1}(\frac{\pi(s-\nu)}{2})$ at $s=\nu,\nu-2,\cdots$ will get canceled by the simple zeroes of $L(s-\nu, \chi)$. 
 Hence the poles of the integrand function in \eqref{jnuprime1} are at $ 2, 4, \cdots, 2N, $ and $\nu+2,\cdots,\nu+2  b_N$ where $b_N=\lfloor \frac{2N+\delta-\nu}{2} \rfloor$, and they are simple. Utilising the fact $|\sin \pi (\sigma+it)|\gg e^{\pi |t|}$ for $|t|\geq 1$, one can see that 
 the integrals along the horizontal lines $[1+d+iT, 2N+\delta+iT]$ and $[2N+\delta-iT, 1+d-iT]$ vanish as $T \rightarrow \infty$. Hence we get
\begin{align}\label{jnuprimekk}
		\frac{1}{2\pi i} \int_{(1+d)}   \frac{\zeta(s)L(s-\nu, \chi) (qx)^{s}}{ \sin\left(\frac{\pi s}{2}\right)  \sin\left(\frac{\pi (s-\nu)}{2}\right)}  \mathrm{d}s=  -\sum_{j=1}^N H_{2j}-\sum_{r=1}^{b_N}\mathcal{H}_{2r}+\frac{1}{2\pi i} \int_{(2N+\delta)}   \frac{\zeta(s)L(s-\nu, \chi) (qx)^{s}}{ \sin\left(\frac{\pi s}{2}\right)  \sin\left(\frac{\pi (s-\nu)}{2}\right)}  \mathrm{d}s, 
		\end{align}
  where $H_{2j}$ is the residue at $s=2j$ given by 
     \begin{align*} 
      H_{2j}=&-\frac{2\zeta(2j)L(2j-\nu,\chi)(qx)^{2j}}{\pi\sin(\frac{\pi\nu}{2})},
      \end{align*}
 for $j=1, 2, \cdots, N$ and $\mathcal{H}_{2r}$ is the residue at $s=\nu+2r$ given by
       \begin{align*}
      \mathcal{H}_{2r} &= 2 \zeta(\nu+2r)L(2r,\chi)\frac{(qx)^{\nu+2r}}{\pi\sin(\frac{\pi\nu}{2})}
=\frac{2}{\pi\sin(\frac{\pi\nu}{2})}\sum_{n=1}^\infty \sigma_{\nu,\chi} (n)  (n^{-1}qx)^{\nu+2r}, 
      \end{align*}
  for $r=1, 2, \cdots, b_N$. In the above expression, we have applied the series representation of function $\zeta(\nu+2r)L(2r,\chi)$ for $r\geq 1$. 
   Now let us evaluate the integral in \eqref{jnuprimekk}:
			\begin{align} \label{jof first case}
			&\frac{1}{2\pi i} \int_{(2N+\delta)}   \frac{\zeta(s) L(s-\nu, \chi) (qx)^{s}}{ \sin\left(\frac{\pi s}{2}\right)  \sin\left(\frac{\pi (s-\nu)}{2}\right)}  \mathrm{d}s \nonumber \\
				&=\sum_{n=1}^{\infty} \sigma_{\nu, \chi}(n) \frac{1}{2\pi i} \int_{(2N+\delta)} \frac{(n^{-1} qx)^{s}}{ \sin\left(\frac{\pi s}{2}\right)  \sin\left(\frac{\pi (s-\nu)}{2}\right)}  \mathrm{d}s  \nonumber \\
				&=\left(\sum_{n< qx}+\sum_{n>qx}\right) \sigma_{\nu, \chi}(n) \frac{1}{2\pi i} \int_{(2N+\delta)} \frac{(n^{-1} qx)^{s}}{ \sin\left(\frac{\pi s}{2}\right)  \sin\left(\frac{\pi (s-\nu)}{2}\right)}  \mathrm{d}s,
    \end{align}
    noting that $qx \notin \mathbb{Z}_+$.
    \par Next, we will first investigate the sum $\sum_{n>qx}$.
			To evaluate this inner line integral in \eqref{jof first case}, we shall use the Cauchy residue theorem with the contour consisting of the lines $ [2N+\delta- i T, 2N+\delta+ i T], [2N+\delta+ i T, 2M+\frac{1}{2}+ i T], [2M+\frac{1}{2} + i T,  2M+\frac{1}{2} - i T], [2M+\frac{1}{2} - iT, 2N+\delta-iT]$ where $M \in \mathbb{N}$ is a large number and $T$ is a large positive number. The poles of the integrand function in \eqref{jof first case} are at $ 2N+2, 2N+4, \cdots, 2M$ and $\nu+2b_N+2,\nu+2b_N+4, \cdots, \nu+2a_M$ where
   $a_M= \lfloor M +\frac{1}{4}-\frac{\nu}{2} \rfloor$, and they are simple. Now taking into account the fact that both the integrals along the horizontal lines $[2M+\frac12-iT, 2N+\delta-iT]$ and $[2N+\delta+iT, 2M+\frac12+iT]$ vanish as $T \rightarrow \infty$, we obtain
			\begin{align*}
				&\frac{1}{2\pi i}  \int_{(2N+\delta)} \frac{(n^{-1} qx)^{s}}{ \sin\left(\frac{\pi s}{2}\right)  \sin\left(\frac{\pi (s-\nu)}{2}\right)}  \mathrm{d}s \nonumber \\
    & =\frac{2 }{\pi \sin\left(\frac{\pi \nu}{2}\right)} \left(\sum_{r=N+1}^{M }    (n^{-1}qx)^{2r}-\sum_{r=b_N+1}^{a_M} (n^{-1}qx)^{\nu+2r}\right) 
				+\frac{1}{2\pi i}  \int_{(2M+\frac{1}{2} )} \frac{(n^{-1} qx)^{s}}{ \sin\left(\frac{\pi s}{2}\right)  \sin\left(\frac{\pi (s-\nu)}{2}\right)}  \mathrm{d}s\\
				&=\frac{2 }{\pi \sin\left(\frac{\pi \nu}{2}\right)} \left(\sum_{r=N+1}^{M }    (n^{-1}qx)^{2r}-\sum_{r=b_N+1}^{a_M} (n^{-1}qx)^{\nu+2r}\right) +O\left((n^{-1}qx)^{2M+\frac{1}{2}}\right). 
			\end{align*}
			 Letting $M \rightarrow \infty$, the error term in the above expression goes to $0$ since $n^{-1}qx <1$. Therefore, we get
			\begin{align*}
				\frac{1}{2\pi i}  \int_{(2N+\delta)} \frac{(n^{-1} qx)^{s}}{ \sin\left(\frac{\pi s}{2}\right)  \sin\left(\frac{\pi (s-\nu)}{2}\right)}  \mathrm{d}s
    &=\frac{2 }{\pi \sin\left(\frac{\pi \nu}{2}\right)} \left(\sum_{r=N+1}^{\infty } (n^{-1}qx)^{2r}-\sum_{r=b_N+1}^{\infty}(n^{-1}qx)^{\nu+2r}\right), 
			\end{align*}
			which in turn will give 
			\begin{align*}
				&\sum_{n>qx}\sigma_{\nu, \chi}(n) \frac{1}{2\pi i} \int_{(2N+\delta)} \frac{(n^{-1} qx)^{s}}{ \sin\left(\frac{\pi s}{2}\right)  \sin\left(\frac{\pi (s-\nu)}{2}\right)}  \mathrm{d}s \nonumber \\
    &=\frac{2 }{\pi \sin\left(\frac{\pi \nu}{2}\right)} \sum_{n>qx} {\sigma}_{\nu, \chi}(n)  \left(\sum_{r=N+1}^{\infty } (n^{-1}qx)^{2r}-\sum_{r=b_N+1}^{\infty}(n^{-1}qx)^{\nu+2r}\right). 
			\end{align*}
   From the above expression, one can deduce
\begin{align}\label{ev1}
				&\sum_{n>qx}\sigma_{\nu, \chi}(n) \frac{1}{2\pi i} \int_{(2N+\delta)} \frac{(n^{-1} qx)^{s}}{ \sin\left(\frac{\pi s}{2}\right)  \sin\left(\frac{\pi (s-\nu)}{2}\right)}  \mathrm{d}s -\frac{2}{\pi\sin(\frac{\pi\nu}{2})}\sum_{n>qx} \sigma_{\nu,\chi} (n) \sum_{r=1}^{b_N} (n^{-1}qx)^{\nu+2r}\nonumber \\
    &=\frac{2 }{\pi \sin\left(\frac{\pi \nu}{2}\right)} \sum_{n>qx} {\sigma}_{\nu, \chi}(n)  \left(\sum_{r=N+1}^{\infty } (n^{-1}qx)^{2r}-\sum_{r=1}^{\infty}(n^{-1}qx)^{\nu+2r}\right)\nonumber\\
    &=\frac{2 }{\pi \sin\left(\frac{\pi \nu}{2}\right)}\sum_{n>qx} {\sigma}_{\nu, \chi}(n)\frac{(qx)^{2N+2} }{ n^{\nu}     } \frac{ \left( n^{\nu-2N}-(qx)^{\nu-2N}\right)}{n^2-q^2x^2}
    \notag\\
  &=\frac{2(qx)^{2N+2} }{  \pi \sin\left(\frac{\pi \nu}{2}\right)   }  \sum_{n>qx}\bar{\sigma}_{-\nu, \chi}(n)   \left(\frac{ n^{\nu-2N}-(qx)^{\nu-2N} }{n^2-q^2x^2}\right).
			\end{align}
 Similarly, by shifting the line of integration to the left, $\sum_{n \leq qx}$ can be evaluated as
			\begin{align}\label{ev2}
				&\sum_{n< qx}\sigma_{\nu, \chi}(n) \frac{1}{2\pi i} \int_{(2N+\delta)} \frac{(n^{-1} qx)^{s}}{ \sin\left(\frac{\pi s}{2}\right)  \sin\left(\frac{\pi (s-\nu)}{2}\right)}  \mathrm{d}s -\frac{2}{\pi\sin(\frac{\pi\nu}{2})}\sum_{n<qx} \sigma_{\nu,\chi} (n) \sum_{r=1}^{b_N} (n^{-1}qx)^{\nu+2r} \nonumber \\
    &=-\frac{2 }{\pi \sin\left(\frac{\pi \nu}{2}\right)} \sum_{n< qx} {\sigma}_{\nu, \chi}(n)  \left(\sum_{r=-N}^{\infty } (n(qx)^{-1})^{2r}-\sum_{r=0}^{\infty}(n(qx)^{-1})^{-\nu+2r}\right)\nonumber \\
    &=-\frac{2 }{\pi \sin\left(\frac{\pi \nu}{2}\right)} \sum_{n< qx} {\sigma}_{\nu, \chi}(n)\left(  \left(\frac{qx}{n}\right)^{2N} \frac{(qx)^2}{(qx)^2-n^2}- \left(\frac{qx}{n}\right)^\nu \frac{(qx)^2}{(qx)^2-n^2} \right)\nonumber \\
      &=\frac{2(qx)^{2N+2} }{  \pi \sin\left(\frac{\pi \nu}{2}\right)   }  \sum_{n<qx}\bar{\sigma}_{-\nu, \chi}(n)   \left(\frac{ n^{\nu-2N}-(qx)^{\nu-2N} }{n^2-q^2x^2}\right).  
     \end{align}
		Now substituting \eqref{ev1} and \eqref{ev2} in \eqref{jof first case},
  \begin{align}\label{JKprime}
				\frac{1}{2\pi i} \int_{(2N+\delta)}   \frac{\zeta(s)L(s-\nu, \chi) (qx)^{s}}{ \sin\left(\frac{\pi s}{2}\right)  \sin\left(\frac{\pi (s-\nu)}{2}\right)}  \mathrm{d}s &=\frac{2}{\pi\sin(\frac{\pi\nu}{2})}\sum_{n=1}^\infty \sigma_{\nu,\chi} (n) \sum_{r=1}^{b_N} (n^{-1}qx)^{\nu+2r} \nonumber \\
   &+\frac{2(qx)^{2N+2} }{  \pi \sin\left(\frac{\pi \nu}{2}\right)   }  \sum_{n=1}^{\infty}\bar{\sigma}_{-\nu, \chi}(n)   \left(\frac{ n^{\nu-2N}-(qx)^{\nu-2N} }{n^2-q^2x^2}\right).
			\end{align}
			Inserting \eqref{JKprime} in \eqref{jnuprimekk} and then simplifying, we obtain
			\begin{align}\label{raaag}
				\frac{1}{2\pi i} \int_{(1+d)}   \frac{\zeta(s)L(s-\nu, \chi) (qx)^{s}}{ \sin\left(\frac{\pi s}{2}\right)  \sin\left(\frac{\pi (s-\nu)}{2}\right)}  \mathrm{d}s&= \frac{2}{ \pi \sin \left(\frac{\pi \nu}{2}\right)}\sum_{j=1}^{N} \zeta(2j)\ L(2j-\nu, \chi)(qx)^{2j} 
				\notag \\&+ \frac{2}{ \pi \sin \left(\frac{\pi \nu}{2}\right)}(qx)^{2N+2}\sum_{n=1}^{\infty}\bar{\sigma}_{-\nu, \chi}(n)  \left(\frac{ n^{\nu-2N}-(qx)^{\nu-2N} }{n^2-q^2x^2}\right).
			\end{align}
		 Combining \eqref{raaag} with \eqref{jnuprime1}, we deduce that
			\begin{align}\label{sad songs}
				  J_{-\nu}^{(\nu)}(X)&=\frac{(2 \pi)^\nu q^{-\nu}}{2\pi x\tau(\chi)  \sin \left(\frac{\pi \nu}{2}\right)}\left\{\sum_{j=1}^{N} \zeta(2j)\ L(2j-\nu, \chi)(qx)^{2j} 
 \right.\notag\\&\left.\ \  
     \hspace{5cm} +(qx)^{2N+2}\sum_{n=1}^{\infty}\bar{\sigma}_{-\nu, \chi}(n)  \left(\frac{ n^{\nu-2N}-(qx)^{\nu-2N} }{n^2-q^2x^2}\right)
    \right\}.
			\end{align}
			Next, by substituting \eqref{sad songs} in \eqref{twistedLL}, one can finish the proof.
		\end{proof}	
  \vspace{.2cm}
The proofs of other remaining theorems in Section \ref{cohen identities...} can be proved similarly. We leave the explanations for the readers.  

		 \section{Proof of Vorono\"i-type Summation formulas }\label{proof of voronoi...}
      In this section, we prove Theorem \ref{vore1}. The proofs of other theorems will be similar, so we will skip the proofs of other theorems. To prove Theorem \ref{vore1}, we will adapt the method introduced by B. C. Berndt, A. Dixit, A. Roy, and A. Zaharescu in \cite{MR3558223}. 
  \begin{proof}[Theorem \rm{\ref{vore1}}][]
  	 Let us recall the Theorem  \ref{evencohen}. One can see that identity \eqref{Cohen1} in Theorem \ref{evencohen} is valid not only for $x>0$ but also for $-\pi< \arg x < \pi$ by analytic continuation. If we set $N=1$ in \eqref{Cohen1}, then the condition $\lfloor\frac{\Re{(\nu)}+1}{2}\rfloor \leq 1$ will
imply that $0\leq \Re{(\nu)} <3$. We consider $0<\Re{(\nu)} < \frac{1}{2}.$
Replace $x$ by $iz/q$ in \eqref{Cohen1} for $-\pi< \arg z < \frac{\pi}{2}$ and then   by $-iz/q$ for $- \frac{\pi}{2}< \arg z < \pi$. Now the common region of the resultant identities is $- \frac{\pi}{2}< \arg z < \frac{\pi}{2}$. So we add the resulting two identities and simplify, in the region $- \frac{\pi}{2}< \arg z  < \frac{\pi}{2}$, to obtain
\begin{align}
  \Lambda(z,\nu)=\Psi_1(z,\nu), \label{e1_1}
\end{align}
 where
\begin{align} \label{maine1}
&\Lambda(z,\nu)= 2z^{-\frac{\nu}{2}}   \    \sum_{n=1}^{\infty}\sigma_{-\nu, \bar{\chi}}(n) \ n^{\nu/2} \left\{ e^{\frac{i\pi \nu}{4}}  K_{\nu}\left(4\pi  e^{\frac{i\pi}{4}}\sqrt{\frac{nz}{q}} \right) 
  +e^{\frac{-i\pi \nu}{4}}  K_{\nu}\left(4\pi  e^{\frac{-i\pi}{4}}\sqrt{\frac{nz}{q}}\ \right)   \right\} , 
  \end{align}
  and
  \begin{align}
  \Psi_1(z,\nu) =  -\frac{ q^{\frac{\nu}{2}}\Gamma(\nu) L(\nu, \bar{\chi})}{(2\pi)^{\nu} } z^{-\nu}
+ \frac{ q^{1-{\frac{\nu}{2}}} }{\pi \tau(\chi) }     \sum_{n=1}^\infty    \frac{\Bar{\sigma}_{-\nu, \chi }(n) }{n^2+z^2} \ z . \label{main_e1}
\end{align}
  Note that $\Psi_1(z,\nu) $ is an analytic function of $z$ in $\mathbb{C}$ except on negative real axis and at $z=i  n $ where $n\in \mathbb{Z}$.
 Hence $\Psi_1(iz,\nu)$ is analytic in $\mathbb{C}$ except on the positive imaginary axis and at $z \in \mathbb{Z}$. 
 Similarly, $\Psi_1(-iz,\nu)$ is analytic in $\mathbb{C}$ except on the negative imaginary  axis and at $z\in \mathbb{Z}$. 
 We deduce $\Psi_1(iz,\nu)+\Psi_1(-iz,\nu)$ is analytic in both the left and right half plane, except possibly when $z$ is an integer. Since
\begin{align*}
\lim_{z\to  \mp n } (z \pm n) \Psi_1(iz,\nu) = \frac{q^{1-{\frac{\nu}{2}}}}{{2\pi i \tau(\chi) } }   \Bar{\sigma}_{-\nu, \chi }(n), \ \ 
  \lim_{z\to \mp n} (z \pm n) \Psi_1(-iz,\nu) = -\frac{q^{1-{\frac{\nu}{2}}}}{{2\pi i \tau(\chi) } }   \Bar{\sigma}_{-\nu, \chi }(n),
\end{align*}
 so we have
\begin{align*}
\lim_{z\to \mp n} (z \pm n) \left( \Psi_1(iz,\nu)+\Psi_1(-iz,\nu) \right)   =0.
\end{align*}
 Hence $\Psi_1(iz,\nu)+\Psi_1(-iz,\nu)$ is analytic in the entire right half plane. From \eqref{main_e1}, we observe that for $z$ lying inside an interval $(a,b)$ on the positive real line not containing an integer, we have
\begin{align}
    \Psi_1(iz,\nu)+\Psi_1(-iz,\nu)=-\frac{2 q^{\frac{\nu}{2}}\Gamma(\nu) L(\nu, \bar{\chi})}{(2\pi)^{\nu} }\cos{\left(\frac{\pi \nu}{2}\right)}\frac{1}{z^{\nu}}   . \label{e1_11}
\end{align}
  Since both sides are analytic in the right half-complex plane as a function of $z$, by analytic continuation, the identity \eqref{e1_11} holds for any $z$ in the right half-plane. Next employing functional equation for $L$-function \eqref{ll(s)} in \eqref{e1_11} and simplifying, we obtain for $- \frac{\pi}{2}< \arg z   < \frac{\pi}{2}$,
\begin{align}
 \Psi_1(iz,\nu)+\Psi_1(-iz,\nu)=-\frac{ q^{1-\frac{\nu}{2}}}{\tau(\chi)} L(1-\nu,\chi) \frac{1}{z^{\nu}}. \label{e1_2}
 \end{align}
Next, Let
   $f$ be an analytic function of $z$ in a closed contour $\mathcal{\gamma^\prime}$ intersecting the real axis in $\alpha$  and $\beta$ where $0<\alpha< \beta ,\ { m-1}< \alpha< { m}, \  { n-1}<\beta<{ n} $ and $m, n \in \mathbb{Z}$. Now $\mathcal{\gamma^\prime}$ consists of two parts $\gamma_1$ and $\gamma_2$ where $\gamma_1$ is the portion of the contour in the upper half-plane, and $\gamma_2$ is the portion corresponding to lower half-plane. Now $\alpha \gamma_1 \beta$ and $\alpha \gamma_2 \beta$ denote the paths from $\alpha $ to $ \beta$ in the upper and lower half planes, respectively.  
By the Cauchy residue theorem, we have 
 \begin{align*}
  \int_{\alpha \gamma_2 \beta \gamma_1 \alpha  }f(z) \Psi_1(iz,\nu) dz =  \frac{q^{1-{\frac{\nu}{2}}}}{{ \tau(\chi) } }\sum_{\alpha   < j < \beta}     \Bar{\sigma}_{-\nu, \chi }(j)f(j),
\end{align*}
where $\frac{q^{1-{\frac{\nu}{2}}}}{{2\pi i \tau(\chi) } }   \Bar{\sigma}_{-\nu, \chi }(j)f(j)$ is the residue of $f(z)\Psi_1(iz,\nu)$  at each integer $j$ where $  \alpha   < j < \beta$. Hence the above expression can be rewritten as
\begin{align}
    \frac{   q^{1-\frac{\nu}{2}}  } {\tau(\chi )}    &\sum_{\alpha<j <\beta}    {\bar{\sigma}_{-\nu, \chi }(j)}  f(j)  
   =\int_{\alpha \gamma_2 \beta   }f(z) \Psi_1(iz,\nu) dz -   \int_{\alpha \gamma_1 \beta   }f(z) \Psi_1(iz,\nu) dz   \notag\\
     &=\int_{\alpha \gamma_2 \beta   }f(z) \Psi_1(iz,\nu) dz +   \int_{\alpha \gamma_1 \beta   }f(z)  \left\{  \Psi_1(-iz,\nu)  +\frac{ q^{1-\frac{\nu}{2}}}{\tau(\chi)}L(1-\nu,\chi)  \frac{1}{z^{\nu}} \right\} dz   ,\label{e1_3}
 \end{align} 
  where in the last step, we used \eqref{e1_2}. Again we make use of the Cauchy residue theorem and  obtain  
 \begin{align}
\frac{ q^{1-\frac{\nu}{2}}}{\tau(\chi)} L(1-\nu,\chi)\int_{\alpha \gamma_1 \beta   }   \frac{f(z)}{z^{\nu}}  \mathrm{d}z=\frac{ q^{1-\frac{\nu}{2}}}{\tau(\chi)} L(1-\nu,\chi) \int_\alpha^\beta    \frac{f(t) }{t^{\nu}}  \mathrm{d}t.\label{M4}
 \end{align}
From \eqref{e1_1}, 
$\Lambda(z,\nu)= \Psi_1(z,\nu)$ for  $- \frac{\pi}{2}< \arg z  < \frac{\pi}{2}$. So it is easy to see that $\Lambda(iz,\nu)= \Psi_1(iz,\nu)$     holds for  $- \pi < \arg z  < 0$,  and $\Lambda(-iz,\nu)= \Psi_1(-iz,\nu)$     holds for  $0 < \arg z  < \pi$. Thus 
\begin{align}\begin{cases}
  \int_{\alpha \gamma_2 \beta   }f(z) \Psi_1(iz,\nu) dz=\int_{\alpha \gamma_2 \beta   }f(z) \Lambda(iz,\nu) dz, \\
  \int_{\alpha \gamma_1 \beta   }f(z) \Psi_1(-iz,\nu) dz = \int_{\alpha \gamma_1 \beta   }f(z) \Lambda(-iz,\nu) dz .  
\end{cases}\label{e1_4}
 \end{align}
 Here we notice that the series $ \Lambda(iz,\nu)$ in \eqref{maine1} is uniformly convergent in compact subintervals of   $- \pi < \arg z  < 0$, and series $\Lambda(-iz,\nu)$
  is uniformly convergent in compact subintervals of $0 < \arg z  < \pi$. Thus, interchanging the order of summation and integration in \eqref{e1_4} and inserting them in \eqref{e1_3} together with \eqref{M4}, we get
   \begin{align*}
  & \frac{   q^{1-\frac{\nu}{2}}  } {\tau(\chi )}    \sum_{\alpha<j <\beta}    {\bar{\sigma}_{-\nu, \chi }(j)} f(j)   =   \frac{ q^{1-\frac{\nu}{2}}}{\tau(\chi)} L(1-\nu,\chi) \int_\alpha^\beta    \frac{f(t) }{t^{\nu}}  \mathrm{d}t
 + 2 \sum_{n=1}^{\infty}\sigma_{-\nu, \bar{\chi }}(n) \ n^{\nu/2}    \\
   & \ \ \ \ \ \   \times  \int_{\alpha \gamma_2 \beta   }f(z)  (iz)^{-\frac{\nu}{2}} \left\{ e^{\frac{i\pi \nu}{4}}   K_{\nu}\left(4\pi  e^{\frac{i\pi}{4}}\sqrt{\frac{inz}{q}}\ \right)  +e^{\frac{-i\pi \nu}{4}}  K_{\nu}\left(4\pi  e^{\frac{-i\pi}{4}}\sqrt{\frac{inz}{q}}\ \right)   \right\} \mathrm{d}z \\
 +   &2 \sum_{n=1}^{\infty}\sigma_{-\nu, \bar{\chi }}(n) \ n^{\nu/2}   \int_{\alpha \gamma_1 \beta   }f(z)  (-iz)^{-\frac{\nu}{2}}  
  \left\{ e^{\frac{i\pi \nu}{4}}   K_{\nu}\left(4\pi  e^{\frac{i\pi}{4}}\sqrt{\frac{-inz}{q}}\ \right) 
  +e^{\frac{-i\pi \nu}{4}}  K_{\nu}\left(4\pi  e^{\frac{-i\pi}{4}}\sqrt{\frac{-inz}{q}}\ \right)   \right\} \mathrm{d}z. 
\end{align*}
  Simplifying we get
  \begin{align*}
  & \frac{   q^{1-\frac{\nu}{2}}  } {\tau(\chi )}    \sum_{\alpha<j <\beta}    {\bar{\sigma}_{-\nu, \chi }(j)} f(j)   =   \frac{ q^{1-\frac{\nu}{2}}}{\tau(\chi)} L(1-\nu,\chi) \int_\alpha^\beta    \frac{f(t) }{t^{\nu}}  \mathrm{d}t  \\
   +& 2 \sum_{n=1}^{\infty}\sigma_{-\nu, \bar{\chi }}(n) \ n^{\nu/2}   \int_{\alpha \gamma_2 \beta   }f(z)  z^{-\frac{\nu}{2}} \left\{   K_{\nu}\left(4\pi  i\sqrt{\frac{nz}{q}}\ \right)  +e^{\frac{-i\pi \nu}{2}}  K_{\nu}\left(4\pi  \sqrt{\frac{nz}{q}}\ \right)   \right\} \mathrm{d}z \\
   &+2 \sum_{n=1}^{\infty}\sigma_{-\nu, \bar{\chi }}(n) \ n^{\nu/2}   \int_{\alpha \gamma_1 \beta   }f(z)  z^{-\frac{\nu}{2}} \left\{ e^{\frac{i\pi \nu}{2}}   K_{\nu}\left(4\pi  \sqrt{\frac{nz}{q}}\ \right)  +  K_{\nu}\left(-4\pi i \sqrt{\frac{nz}{q}}\ \right)    \right\} \mathrm{d}z .
    \end{align*}
   Employing the residue theorem again, this time for each of the integrals inside the two sums, and simplifying, we obtain
 \begin{align}
    &\frac{   q^{1-\frac{\nu}{2}}  } {\tau(\chi )}    \sum_{\alpha<j <\beta}    {\bar{\sigma}_{-\nu, \chi }(j)} f(j)   =   \frac{ q^{1-\frac{\nu}{2}}}{\tau(\chi)} L(1-\nu,\chi) \int_\alpha^\beta    \frac{f(t) }{t^{\nu}}  \mathrm{d}t+2 \sum_{n=1}^{\infty}\sigma_{-\nu, \bar{\chi }}(n) \ n^{\nu/2}  \notag\\ 
&   \times  \int_{\alpha} ^{  \beta   }f(t)  t^{-\frac{\nu}{2}}   \left\{   K_{\nu}\left(4\pi  i\sqrt{\frac{nt}{q}}\ \right)   
+   K_{\nu}\left(-4\pi  i\sqrt{\frac{nt}{q}}\ \right)+ 2 \cos\left(\frac{\pi \nu}{2} \right)  K_{\nu}\left(4\pi  \sqrt{\frac{nt}{q}}\ \right)  \right\} \mathrm{d}t.   
 \label{e1_5}
   \end{align}
    Here by \cite[p.~ 848, equation (7.15)]{MR3558223}, we have
     \begin{align}
     K_{\nu}(ix)+ K_{\nu}(-ix) 
   = -\pi \left( J_\nu(x) \sin(\frac{\pi \nu}{2})+       Y_\nu(x) \cos(\frac{\pi \nu}{2}) \right),\label{Z101}
\end{align} 
where $J_\nu$ and $Y_\nu$ are the Bessel functions defined in  \eqref{J_bessel} and \eqref{Y_bessel}, respectively. 
   Now, we replace $x$ by $4\pi  \sqrt{nt/q} $ in \eqref{Z101} and substitute in \eqref{e1_5}, to get
  the desired result.
 \end{proof}

 		\section{Acknowledgement}
We thank Prof Atul Dixit for his valuable input. Research of the second author was supported by the University Grants Commission, Department of Higher Education, Government of India, under NTA Ref. no. 191620205105. They are grateful to the referee for several helpful suggestions and comments.
\bibliographystyle{IEEEtran}
		 	 	\bibliography{bibliography.bib}
		\end{document}